\documentclass[10pt]{amsart}
\usepackage[all]{xy} 
\usepackage{verbatim}
\usepackage{pstricks}
\usepackage{enumitem}
\usepackage{amsfonts, amssymb}

\newtheorem{thm}{Theorem}[section]
\newtheorem{cor}[thm]{Corollary}
\newtheorem{prop}[thm]{Proposition}
\newtheorem{lem}[thm]{Lemma}
\theoremstyle{definition}
\newtheorem{defn}[thm]{Definition}
\newtheorem{ex}[thm]{Example}

\newtheorem{rmk}[thm]{Remark}
\newtheorem*{ack}{Acknowledgments}

\numberwithin{equation}{section}

\newcommand{\even}{\textrm{even}}
\newcommand{\odd}{\textrm{odd}}

\newcommand{\U}{U}

\newcommand{\II}{\mathcal{I}}

\newcommand{\Tr}{\mathrm{Tr}}

\newcommand{\B}{BU\times \Z}%{\mathcal B}
\newcommand{\R}{\mathbb R}
\newcommand{\C}{\mathbb C}
\newcommand{\Z}{\mathbb {Z}}
\newcommand{\cc}{{cpt}}

\newcommand{\N}{\mathbb {N}}

\newcommand{\E}{E}%{\mathrm{E}}

\newcommand{\Om}{\Omega}

\title[Differential $K$-theory as equivalence classes...]{Differential $K$-theory as equivalence classes of maps to Grassmannians and unitary groups}

\author[T.~Tradler]{Thomas~Tradler}
  \address{Thomas Tradler,
  Department of Mathematics, New York City College of Technology, City University of New York, 300 Jay Street, Brooklyn, NY 11201}
  \email{ttradler@citytech.cuny.edu}

\author[S.~Wilson]{Scott O. Wilson}
  \address{Scott O. Wilson, Department of Mathematics, Queens College, City University of New York, 65-30 Kissena Blvd., Flushing, NY 11367}
  \email{scott.wilson@qc.cuny.edu}

\author[M.~Zeinalian]{Mahmoud~Zeinalian}
  \address{Mahmoud Zeinalian, Department of Mathematics, LIU Post, Long Island University, 720 Northern Boulevard, Brookville, NY 11548, USA} 
  \email{mzeinalian@liu.edu}

\begin{document}

\begin{abstract}
We construct a model of differential $K$-theory, using the geometrically defined Chern forms, whose cocycles are certain equivalence classes of maps into the Grassmannians and unitary groups. In particular, we produce the circle-integration maps for these models using classical homotopy-theoretic constructions, by incorporating  certain differential forms which reconcile the incompatibility between these even and odd Chern forms. By the uniqueness theorem of Bunke and Schick, this model agrees with the spectrum-based models in the literature whose abstract Chern cocycles are compatible with the delooping maps on the nose.\end{abstract}
%\allowdisplaybreaks

\maketitle

\setcounter{tocdepth}{2}
\tableofcontents

\section{Introduction}

Differential cohomology theories provide a refinement of cohomology theories which incorporate additional geometric information. A historically important first example is \emph{differential ordinary cohomology}, where elements of ordinary cohomology, such as the first Chern class, are enriched by the additional geometric data, such as a line bundle with connection. These stem from original discussions of Cheeger-Simons \cite{CS} concerning differential characters, Harvey and  Lawson \cite{HL} concerning spark complexes, and  Deligne cohomology \cite{D}. 

Hopkins and Singer in \cite{HS}  showed that every cohomology theory has a differential refinement, and  gave a construction of a differential refinement of a cohomology theory, as a fibered product of the cohomology theory with differential forms or cocycles.
In  \cite{BS2}, Bunke and Schick give axioms for differential cohomology theories, and in particular show that having both a differential extension, and an $S^1$-integration map, often determines a  differential cohomology theory uniquely, up to a unique natural isomorphism, whereas the differential extension alone does not. For example, this is the case for differential $K$-theory, which is the topic of this paper.  We refer the reader to \cite{BS3} for a survey of several applications of differential $K$-theory to mathematics and physics. 

In \cite{BS} Bunke and Schick construct a geometric model of differential $K$-theory, based on ideas of local index theory, using families of Dirac operators. This model is advantageous for its connections and applications to index theory. Subsequent work by Simons and Sullivan \cite{SS} constructed a differential extension of the even degree part of $K$-theory as roughly the Grothendieck group of vector bundles \emph{with connection}. This even degree construction agrees with the even part of differential $K$-theory. A particularly nice feature of this model is that the construction does not require the initial data of differential forms;  vector bundles with connection, and no additional geometric data, naturally determine even degree differential $K$-theory. More recently, it has been proposed that differential cohomology theories are perhaps best defined in terms of sheaves of spectra on manifolds \cite{BNV}, \cite{HS}, and has been generalized to incorporate cohesion in \cite{USc}. 

In section 2 we recall the definitions needed for this paper. We show that for a fixed differential cohomology theory, it follows from the axioms that the underlying functor is well defined on a non-trivial quotient of the category of smooth manifolds, where a morphism is an equivalence class of smooth maps, determined by the fiber integration of ``Chern forms'' (see Remark \ref{rmk:quotientcat}). This provides a refinement of the homotopy category, sensitive to the theory. For differential $K$-theory, we denote this category by $Smooth_{\hat K}$, where two smooth maps $f_0, f_1 :M \to N$ are equivalent if and only if there is a smooth homotopy $f_t$ such that $\int_I f_t^* X$ is a Chern form on $M$, whenever $X$ is Chern form on $N$.

In section 3 we construct models for an even and odd differential extensions of $K$-theory. The odd case is similar to the authors' work in
\cite{TWZ3}, using \emph{only} the mapping space of the stable unitary group $U$, though the space $U$ is changed slightly here in order to build the $S^1$-integration maps for the theory. The subsequently described even differential extension of $K$-theory, constructed using a specific model of $\B$ as the target, is entirely new. In particular, we show that the data of differential forms is not required as input to the theory, but rather is derivable from the geometry of the model for $\B$ used here. In light of the constructions in \cite{SS} mentioned above, the idea that this is possible is perhaps not too surprising, but  there are several technical difficulties in constructing a suitable model of $\B$, with explicit requisite structure (\emph{e.g.} the monoid structure) which are overcome here in detail. 
 Moreover, there is a natural map from the even part of our model to the Simons-Sullivan model of differential $K$-theory which, using a strengthening of the Narasimhan Ramanan theorem, we prove to be an isomorphism.
 
A nice consequence of the even and odd degree models given here is that the set of cocycles is a rather small and familiar object from topology. These models are both expected to have interesting connections with higher degree gerbes and  field theories.

In section 4 we construct explicit $S^1$-integration maps for the theory. In particular this implies that the definitions given in \cite{TWZ3} do indeed define the odd degree part of differential $K$-theory, which was previously obtained by Hekmati, Murray, Schlegel, and Vozzo, in \cite{HMSV}, using different methods. 

From the onset in this paper, we insist on using the even and odd degree Chern forms that appear in geometry, as opposed to abstract cocycle representatives.  To build the $S^1$-integration map, we must then construct explicit differential forms which measure the 
failure of  compatibility between these Chern forms, the $S^1$-integration of differential forms, and the natural candidates for 
$S^1$-integration coming from the transformations on mapping spaces induced by explicit homotopy equivalences $\B \to \Om U$ and
 $\Om (\B) \to U$.  The existence of such cochains has been proved previously using abstract cohomological arguments \cite{HS}, but to our knowledge these explicit expressions are new.
 
It is natural to ask how these results relate to the possible \emph{representability} of differential $K$-theory. Of course, one could not possibly represent this functor by a finite dimensional manifold, and on the other hand, any  functor extends to a representable functor on some category (\emph{e.g.} the functor category, by Yoneda's Lemma).  One could ask for representability in some geometric category such as diffeological spaces. 
 
 The results here produce \emph{by definition} set bijections 
\[
\hat K^0(M) = Hom(M , \B) \quad \quad \hat K^{-1}(M) = Hom(M , U)
\]
where the Hom sets are defined as $CS$-equivalence classes of maps into the respective target: two maps 
$f_0$ and $f_1$ are equivalent if and only if there is a smooth homotopy $f_t$ such that $\int_I f_t^* Ch$ is a Chern form on $M$, 
where $Ch$ denotes \emph{the} Chern form on $BU$ or $U$ defined below, respectively. 
One might hope that the equivalence relation defining $Hom(M,N)$ in $Smooth_{\hat K}$ adapted to the cases $N=BU$ and $N=U$ would give differential $K$-theory as well, but it turns out that set is too large, even for $M=pt$. We are resolved then that
 differential $K$-theory is equal to equivalence classes of maps determined by fiber integration of 
\emph{the} Chern form on targets $BU$ and $U$, and is functorial with respect to equivalence classes of maps determined by fiber integration of  \emph{any} Chern form on target $N$.

\begin{ack}
The first and second authors were supported in part by grants from The City University of New York PSC-CUNY Research Award Program. The third author was partially supported by the NSF grant DMS-1309099 and would like to thank the Max Planck Institute and the Hausdorff Institute for Mathematics for their support and hospitality during his visits. 
All three authors gratefully acknowledge support from the Simons Center for Geometry and Physics, Stony Brook University, at which some of the research for this paper was performed.
\end{ack}

\section{Differential Extensions of $K$-theory}\label{SEC:Diff-Ext-of-K}

In this section, we recall  the definition of a differential extension of $K$-theory with $S^1$-integration map from Bunke and Schick \cite{BS2} and \cite{BS3}, as well as a uniqueness theorem for differential $K$-theory, \cite[Theorem 3.3]{BS3}, which will be used to identify the  construction given here with differential $K$-theory. 
Denote by $K^*(M)$ the complex $K$-theory of a manifold $M$ (possibly with corners). Recall that the Chern character $Ch$ induces a ($\Z_2$-graded) natural transformation $[Ch]:K^*(M)\to  H^*(M)$. (For an explicit construction of these, see the next section.)

\begin{defn}[Definition 2.1, \cite{BS3}] \label{defn:diffext}
A \emph{differential extension of $K$-theory} is a quadruple $(\hat K, R, I , a)$, where $\hat K$ is a contravariant functor from the category of compact smooth manifolds to the category of $\Z_2$-graded abelian groups, and $R$, $I$, $a$ are natural transformations
\begin{enumerate}
\item $R : \hat K^* (M) \to \Om^*_{cl} (M;\R)$ 
\item $I :  \hat K^* (M)  \to K^*(M)$
\item $a: \Om^{*}(M;\R) / Im(d) \to \hat K^{*+1}(M)$
\end{enumerate}
such that
\begin{enumerate}[resume]
\item\label{EQU:diff-k-theory-diamond}
The following diagram commutes
\[
\xymatrix{
 & K^{*}(M) \ar[rd]^{[Ch]} & \\
 \hat K^{*}(M) \ar[ru]^I \ar [rd]^{R} & &H^{*}(M) \\
 & \Om^{*}_{cl} (M) \ar[ru]_-{\textrm{deRham}} &
}
\]
\item $R \circ a = d$
\item The following sequence is exact
\[
\xymatrix{
K^{*-1}(M) \ar[r]^-{[Ch]} &  \Om^{*-1}(M;\R) / Im(d) \ar[r]^-{a} & \hat K^{*}(M) \ar[r]^{I} & K^*(M) \ar[r]^{\quad 0} & 0
} 
\]
\end{enumerate}
\end{defn}

\begin{rmk}
The diagram in condition \eqref{EQU:diff-k-theory-diamond} fits into the following \emph{character diagram} for differential $K$-theory, where the top and bottom sequences that connect $H^{*-1}(M)$ with $H^*(M)$ form long exact sequences. \[
\xymatrix{
 0\ar[rd] &  &  &  & 0  \\
 & Ker(R)\ar[rr]\ar[rd] &  & K^*(M) \ar[rd]^{[Ch]}\ar[ru]&   \\
 H^{*-1}(M)\ar[ru]\ar[rd] &  & \hat K^{*}(M) \ar[ru]^{I}\ar[rd]^{R} & & H^{*}(M) \\
 & \frac{\Om^{*-1}(M;\R)}{Im([Ch])} \ar[rr]^{d}\ar[ru]^{a} &  & Im(R)  \ar[rd]\ar[ru] &  \\
 0\ar[ru] &  &  &  & 0  
}
\]
The group $Ker(R)$ has been considered by Karoubi \cite{K}, and subsequently by Lott \cite{L}, as $\R/\Z$-valued $K$-theory.
Bunke and Schick have shown that $Ker(R)$, the so-called flat-theory, is a homotopy invariant, and in fact extends to a cohomology theory (Theorem 7.11 \cite{BS2}).
\end{rmk}

The following homotopy lemma is given in \cite[Lemma 5.1]{BS2}.

\begin{lem} \label{lem:hmptyformula}
Suppose $f_t : M \times I \to N$ is a smooth homotopy. Then for any $x \in \hat K^*(N)$ we have
\[
f_1^* x - f_0^* x = a \left( \int_I f_t^* R(x) \right)
\]
 \end{lem}
 
\begin{defn} \label{defn:SmoothK}
For any differential extension $\hat K$ of $K$, there is a category $Smooth_{\hat K}$ with objects smooth manifolds, and morphisms given by equivalence classes of smooth maps, where 
two smooth maps $f_0, f_1: M \to N$ are equivalent if there is a smooth homotopy $f_t: M \times I \to N$, such that
\[
\int_I f_t^* X \in \Om^*(M) \in Im(R) \quad \textrm{whenever} \quad  X \in Im(R).
\]
\end{defn}

Let us check this indeed defines a category. The relation is easily seen to be an equivalence relation, since integration along the fiber is additive, and image of $R$ is a subgroup. The only remaining item to check is that composition of morphisms is well defined. Suppose $f_0 \sim f_1$, via some smooth homotopy $f_t : M \times I \to N$, and $g_0 \sim g_1$, via some smooth homotopy $g_t : N \times I \to P$.
Then $h_t : M \times I \to P$ defined by $h_t = (g_t \circ f_0) * (g_1 \circ f_t )$ is a homotopy from $g_0 \circ f_0$ to $g_1 \circ f_1$,
and, if $X \in Im(R)$ then  
\[
\int_I h_t^*X =  f_0^* \int_I   g_t^* X + \int_I  f_t ^*(g_1^* X) .
\]
For the first summand on the right hand side, $\int_I   g_t^* X$ is in the image of $R$ by assumption on $g_t$, and therefore
so is $f_0^* \int_I   g_t^* X $ by  naturality of $\hat K$. Also, $X \in Im(R)$ implies 
$g_1^* X \in Im(R)$ by naturality of $\hat K$, so $ \int_I  f_t ^*(g_1^* X)$ is in the image of $R$ by assumption on $f_t$. This shows $g_0 \circ f_0$ is equivalent to $g_1 \circ f_1$, and we are done.

\begin{cor} \label{cor:Konquotient}
For any differential extension $\hat K$ of $K$, the underlying functor $\hat K$ is well defined on $Smooth_{\hat K}$. 
\end{cor}

\begin{proof}
It suffices to show that if two smooth maps $f_0, f_1: M \to N $ are equivalent, then $f_0^* = f_1^*$. 
Suppose $f_t: M \times I \to N$ is a smooth homotopy such that 
\[
\int_I f_t^* X \in \Om^*(M) \in Im(R)
\]
whenever $ X \in Im(R)$. Then, by the homotopy lemma \ref{lem:hmptyformula}, for any $x \in \hat K(N)$ we have
\[
f_1^* x - f_0^* x  = a \left( \int_I f_t^* R(x) \right).
\]
By commutative diagram of $\hat K^*$ we have that $Im([Ch])$ is equal to $Im(R)$ mod exact, and by the exact sequence we have 
that $Ker(a) = Im([Ch])$, so the right hand side vanishes, and therefore $f_0^* = f_1^*$. 
\end{proof}

\begin{rmk} \label{rmk:quotientcat}
One can similarly define a differential extension $\hat E$ of any cohomology theory $E$, \cite{BS2}. Just as in Definition
\ref{defn:SmoothK} above, one obtains a quotient of the category $Smooth$ of smooth manifolds, by declaring two 
smooth maps $f_0, f_1: M \to N$ to be equivalent if there is a smooth homotopy $f_t: M \times I \to N$, such that
\[
\int_I f_t^* X \in \Om^*(M) \in Im(R) \quad \textrm{whenever} \quad  X \in Im(R).
\]
 where $R: \hat E \to \Om$ is given in the differential extension. It follows just as in Corollary \ref{cor:Konquotient}, using only the axioms of a differential extension,  that $\hat E$ is well defined on this category. 
\end{rmk}

\subsection{$S^1$-Integration}

 We next recall that the $S^1$-integration map in $K$-theory, and then the notion of an $S^1$-integration for a differential extension of $K$-theory.

\begin{defn} \label{def;S1intKtheory}
The inclusion $j: M \to M \times S^1$ via the basepoint and the projection $p: M \times S^1 \to M$
induce a direct sum decomposition 
\[
K^*(M \times S^1) \cong Im(p^*) \oplus Ker(j^*),
\]
where the map is given by $\alpha \mapsto (p^* j^*\alpha, \alpha - p^* j^*\alpha)$. 
The map $q: M \times S^1 \to \Sigma M_+$ induces an isomorphism 
$q^*: \tilde K^* (\Sigma M_+) \to Ker(j^*)$,
which is composed with the suspension isomorphism $\sigma: K^{*-1} (M) \to  \tilde K^* (\Sigma M_+)$
in the following way to define the $S^1$-integration map in $K$-theory
\[
\xymatrix{
\int_{S^1} : K^*(M \times S^1)  \ar[r]^-{pr} & Ker(j^*) \ar[r]^-{(q^*)^{-1}} & \tilde K^* (\Sigma M_+) \ar[r]^-{\sigma^{-1}} & K^{*-1}(M).
}
\]
\end{defn}

\begin{defn}[Definition 1.3, \cite{BS3}] \label{defn:S^1int}
Let $(\hat K, R, I , a)$ be a differential extension of $K$-theory.
An $S^1$ integration map is by definition a natural transformation of functors $\II : \hat K^{*+1} ( - \times S^1)
\to \hat K^*(-)$ satisfying the following three properties:
\begin{enumerate}
\item $\II \circ ( id \times r)^* = - \II$ where $r :S^1 \to S^1$ is given by $r(z) = \bar z$.
\item $\II \circ p^* = 0$ where $p : M \times S^1 \to M$ is projection.
\item The following diagram commutes for all manifolds $M$
\[
\xymatrix{
\Om^*(M \times S^1) /Im(d) \ar[r]^-{a} \ar[d]^{\int_{S^1}} & \hat K^{*+1}(M \times S^1) \ar[r]^{I} \ar[d]^{\II} \ar@/^2pc/[rrr]^{R} & K^{*+1}(M \times S^1) \ar[d]^{\int_{S^1}} & & \Om^*_{cl}(M \times S^1) \ar[d]^{\int_{S^1}} \\
\Om^{*-1}(M) /Im(d) \ar[r]^-{a} & \hat K^{*}(M ) \ar[r]^{I} \ar@/_2pc/[rrr]^{R}& K^{*}(M ) & & \Om^{*-1}_{cl}(M )
}
\]
where the maps $\int_{S^1}$ on differential forms are integration over the fiber $S^1$ and the map $\int_{S^1} : K^{*+1}(M \times S^1) \to K^*(M)$ is the $S^1$-integration map in $K$-theory.
\end{enumerate}
\end{defn}

Bunke and Schick have shown these structure uniquely determine differential $K$ theory. The following
theorem follows from \cite{BS2}, but was succinctly stated as such in Theorem 3.3 of \cite{BS3}.

\begin{thm}[\cite{BS2}, \cite{BS3}]\label{thm:uniqueness-of-diff-k-theory}
Let $(\hat K, R, I, a,\II)$ and $(\hat K', R', I', a',\II')$ be two differential extensions of complex $K$-theory with $S^1$-integrations. Then there is a unique natural isomorphism $\hat K\to \hat K'$, compatible with all the given structures. 
\end{thm}

\section{An Explicit Differential Extension of $K$-Theory}\label{SEC:Diff-K-Theory}

In this section we give an explicit differential extension of $K$-theory by taking certain equivalence classes of maps into the infinite unitary group $U$ or $BU\times \Z$.  
The odd part of the differential extension given here is essentially the one developed in \cite{TWZ3}, whereas the even part is to our knowledge new.
We emphasize that this approach does not require the additional data of differential forms as input.
 We rely on specific models for  $U$ and $BU\times \Z$, with explicit monoid structures, universal Chern forms, and a Bott periodicity map.

\subsection{Model for $U$ and the odd differential extension}

We first recall some  specific constructions for the unitary group $U$.

Let $\C_{-\infty}^{\infty}=span(\{e_i\}_{i \in \Z})$ be the complex vector space given by the span of vectors $\{e_i\}_{i \in \Z}$. We will also denote this by $\C_{-\infty}^{\infty}=\C^\Z_\cc$, the space of maps from $\Z$ to $\C$ with compact support. It will be useful to adopt the following notation
\begin{equation*}
\C^q_p  = span\{e_i | p \leq i < q \},  \quad \C^q_{-\infty}  = span\{e_i | i < q \} , \quad
\C_p^\infty  = span\{e_i | p \leq i  \}.
\end{equation*}
There is an inner product on $\C_p^q$ given by $<e_i,e_j>=\delta_{i,j}$. Note that we have the inclusions $\C_p^q\subset \C_p^{q+1}$ and $\C_p^q\subset \C_{p-1}^q$.

\begin{defn} \label{defn;U}
Let $U_p^q$ be the Lie group of unitary operators $A$ on $\C_p^q$, \emph{i.e.} the space of linear maps $A:\C_p^q\to \C_p^q$ such that $<A(x),A(y)>=<x,y>$ for all $x,y\in \C_p^q$. 
We have the inclusions $U_0^0\subset U_{-1}^1\subset U_{-2}^2\subset\dots$ given by
\[ A\in U_{-p}^p \quad \mapsto \quad 
Id_{\C}\oplus A\oplus Id_\C \in U_{-(p+1)}^{p+1}.
\]
Let
\[
U=U(\C_{-\infty}^\infty)=\bigcup_{p\geq 0} U_{-p}^p
\]
 be the stable infinite unitary group on $\C_{-\infty}^\infty=\C^\Z_\cc$. Equivalently, $U$ is the group of unitary operators on $\C_{-\infty}^{\infty}$ whose difference from the identity $Id$ of $\C_{-\infty}^\infty$ has finite rank. 
 We put the final topology on $U$, that is, a subset $V\subset U$ is open if and only if the space $V\cap U_{-p}^p$ is open in $U_{-p}^p$ for all $p$.
\end{defn}

We remark that this definition of the stable unitary group is isomorphic to the group $U(\C_{0}^\infty)$ of unitary operators on $\C_{0}^\infty=\C^{\N_0}_\cc$ whose difference from the identity on $\C^{\N_0}_\cc$ has finite rank. In fact, any isomorphism $\rho:\C^{\N_0}_\cc\to \C^\Z_\cc$ permuting the ordered basis $\{e_i\}_{i}$ induces an isomorphism of the unitary groups
\[
\tilde\rho: U(\C^\Z_\cc)\to U(\C_\cc^{\N_0}), \quad \tilde \rho(A)=\rho^{-1}\circ A\circ\rho.
\]

In \cite{TWZ3} the author's construct a differential extension of odd $K$-theory using the group $U(\C_{0}^\infty)$. These constructions work equally well with the group $U=U(\C_{-\infty}^\infty)$ defined above, but the model using $\C_{-\infty}^\infty$ fits better with our discussion of $BU \times \Z$ below, so we will review it here for completeness.

In our discussion below, we will need to consider the smooth structures and  DeRham forms on $U$ and $BU\times \Z$. Since $U$ and $BU\times \Z$ are not finite dimensional manifolds, we use the more general notion of plots and differential forms given by plots, see \emph{e.g.} \cite[Definitions 1.2.1 and 1.2.2]{C}. Since $U$ and $BU \times \Z$ are filtered by finite dimensional smooth manifolds,
it is sufficient (and in fact equivalent) to consider differential forms on each finite manifold that are compatible with the filtration, thus justifying the following definition.

A $k$-form $\alpha$ on $U$ is given by a sequence of forms $\{\alpha_p\in\Om^k(U_{-p}^p)\}_{p\geq 0}$
 such that, for all $p \geq 0$, we have 
\[  incl_p^*(\alpha_{p+1})=\alpha_{p}, \]
where $incl_p:U_{-p}^p\subset U_{-(p+1)}^{p+1}$ is the inclusion.

\begin{defn} \label{defn;ChU}
We define the universal odd Chern form $Ch \in \Om^{\textrm{odd}}_{cl}(U)$ by
\begin{equation}\label{EQ-Def:odd-Ch}
Ch:=Tr\sum_{n\geq 0}\frac{(-1)^n}{(2\pi i)^{n+1}}\frac{n!}{(2n+1)!}\omega^{2n+1},
\end{equation}
Here  $\omega=\{\omega_p\}$ where $\omega_p\in \Om^1(U_{-p}^p)$ is the left invariant Lie algebra valued $1$-form on the unitary groups $U_{-p}^p$. Note that $\omega$ is well defined since the inclusions $incl_p:U_{-p}^p\to U_{-(p+1)}^{p+1}$ are group homomorphism and thus $incl_p^*(\omega_{p+1})=\omega_p$. In particular, for any smooth map $g: M \to U$ we have an odd degree closed form 
\[
Ch(g) := g^*(Ch) \in \Om^{\textrm{odd}}_{cl}(M).
\] 
\end{defn}

\begin{defn} \label{EQ-Def:even-CS}
The odd Chern form induces an associated transgression form $CS \in \Om^{\textrm{even}}(PU)$ on the path space $PU$ of $U$, defined 
\begin{equation}
CS:= \int_{t\in I} ev_t^*(Ch)
\end{equation}
where $ev_t : PU \to U$ is evaluation of a path at time $t$.
\end{defn}

By Stokes' theorem we have for $\gamma \in PU$ that 
\begin{equation}\label{EQ:even-CS-Stokes}
d CS_\gamma = ev^*_1(Ch) - ev^*_0(Ch) .
\end{equation}
In particular, for a map $g_t : M\times I \to U$, \emph{i.e.} $g_t:M\to PU$, we have 
\[
CS(g_t) := g_t^*(CS) = \int_{t\in I} Ch(g_t)
\]
satisfies $d CS(g_t) = Ch(g_1) - Ch(g_0)$. Note that this implies that the form $CS \in \Om^{\textrm{even}}(\Om U)$, obtained by restriction to the based loop space $\Om U$ of $U$, is closed.

\begin{defn} \label{defn:CSequivodd}
Two maps $g_0, g_1 : M \to U$ are $CS$-equivalent if there is a smooth homotopy $g_t : M \times I \to U$ such that  $CS(g_t) \in \Om^{\even}(M)$ is \emph{exact}.
\end{defn}

\begin{rmk}
It follows from Theorem \ref{THM:3-statements} below that the previously introduced equivalence relation is the same as the 
equivalence relation: two maps $g_0, g_1 : M \to U$ are $CS$-equivalent if there is a smooth homotopy $g_t : M \times I \to U$ such that  $CS(g_t) \in \Om^{\even}(M)$ is an even degree \emph{Chern form} on $M$ (see Definition \ref{defnCh}).
\end{rmk}

We now wish to introduce a monoid structure on $U$, making the set of $CS$-equivalence class of maps $g:M \to U$ into
an abelian group. This was achieved in \cite{TWZ3} using the elementary block sum operation $\oplus$
on $U(\C_{0}^\infty)$ defined below, which works perfectly well for constructing both the odd $K$-theory group of $M$, and the differential extension in \cite{TWZ3}. 
But, since this operation is discontinuous on $U$, it is cumbersome to use this operation to build $S^1$-integration maps that are  
group homomorphisms induced by Bott Periodicity maps between $BU \times \Z$ and $\Om U$.

Instead, we will introduce another monoid operation on $U$, denoted by $\boxplus$, which is continuous and in fact smooth on plots. As we will show, these two operations will induce the same operations on the $CS$-equivalence classes, \emph{c.f.} Lemma \ref{thm:CSequivUsums1}. First let us recall the elementary block sum.

\begin{rmk} \label{Usum}
In \cite[Section 2]{TWZ3} the block sum operation on $U(\C_\cc^{\N_0})$ is defined as follows. If $A, B\in U(\C_\cc^{\N_0})$ are given by $A=A_0^k\oplus Id_{k+1}^\infty$ and $B=B_0^\ell\oplus Id_{\ell+1}^\infty$, then $A\oplus B=A_0^k\oplus B_0^\ell\oplus Id_{k+\ell+1}^\infty$.  Note that this definition depends on the chosen integer $k$, and there is no consistent choice for $k$ to make this block sum into a continuous operation on all of $U$. However, we may remedy this below by a shuffle sum operation $\boxplus$, which does not depend on any choice.
\end{rmk}

We now  define a shuffle sum operation $\boxplus$ which is continuous on $U$ but not associative.

\begin{defn} \label{Usum2}
Consider the inclusion $U(\C_\cc^\Z)\times U(\C_\cc^\Z)\hookrightarrow U(\C_\cc^\Z\oplus\C_\cc^\Z)=U(\C_\cc^{\Z\sqcup \Z})$. 
For any isomorphism $\rho:\C^\Z_\cc\to \C^{\Z\sqcup \Z}_\cc$ given by relabeling the basis elements $e_i$ of $\C^\Z_\cc$ and $\C^{\Z\sqcup \Z}_\cc$, there is an induced isomorphism of the unitary groups
\[
\tilde\rho: U(\C_\cc^{\Z\sqcup\Z})\to U(\C_\cc^{\Z})=U, \quad \tilde \rho(A)=\rho^{-1} \circ A\circ \rho.
\]
We choose $\rho$ to be the shuffle map $\rho_{sh}:\C^\Z_\cc\to \C^{\Z\sqcup \Z}_\cc$, which maps a basis element $e_i$ with an even index to $\rho_{sh}(e_{2k})=e_k$ into the first $\Z$ component, and a basis element with an odd index $\rho_{sh}(e_{2k+1})=e'_k$ into the second $\Z$ component. With this, we define the (shuffle) block sum as
\[
\boxplus:U(\C_\cc^\Z)\times U(\C_\cc^\Z)\stackrel{incl}{\hookrightarrow} U(\C_\cc^{\Z\sqcup \Z})\stackrel{\tilde\rho_{sh}}{\to} U(\C_\cc^\Z).
\]
Note, that $\boxplus$ is a composition of two continuous maps, and thus continuous.
\end{defn}

\begin{rmk} \label{rmk;CHCSUadd}
With respect to the two operations above, we have $Ch(f \boxplus g) = Ch(f) + Ch(g)$ and $CS(f_t \boxplus g_t) = CS(f_t) + CS(g_t)$, because in both cases the trace is additive. Moreover, $CS(f_t \circ g_t) = CS(f_t) + CS(g_t)$.
\end{rmk}

It follows that the set of  $CS$-equivalence classes of maps $M \to U$ inherits two binary operations, and is a monoid under the operation induced by $\oplus$. The following lemma shows that these two operations on $CS$-equivalence classes are in fact equal.

\begin{lem}\label{thm:CSequivUsums1}
Let $f,g: M \to U$ be smooth. There is a smooth homotopy $h_t : M \times I \to U$ satisfying
$h_0 = f \oplus g$, $h_1 = f \boxplus g$ and $CS(h_t) = 0$. In particular,
$f \oplus g$ is $CS$-equivalent to $f \boxplus g$,  so that $[f \oplus g] =[f \boxplus g]$.
\end{lem}
\begin{proof}
We may assume $f, g: M \to U_{-n}^n$ for some $n \in \N$, so that $f\oplus g$ and $f \boxplus g$ are maps from $M$ to $U_{-2n}^{2n}$. The maps $f\oplus g$ and $f \boxplus g$ differ only by an automorphism of $U_{-2n}^{2n}$ induced by a permutation of coordinates, and since an arbitrary permutation of the coordinates of $\C_{-2n}^{2n}$ may be obtained as a composition of transpositions of adjacent coordinates, it suffices to show there is a path $S_t$ from $A \oplus B$  to $B \oplus A$, for maps $A,B : M \to U$ such that $CS(S_t) = 0$. This is proved in Lemma 3.6 of \cite{TWZ3}. 
\end{proof}

In \cite{TWZ3} it is shown that $\oplus$ in fact induces an abelian group structure on the set of $CS$-equivalence classes of maps $g:M\to U$, with the inverse of $[g]$  given by $[g^{-1}]$. So, by the above Lemma, there is also the same abelian group structure induced by $\boxplus$.

This defines a contravariant functor from compact manifolds to  abelian groups, which we denote by $M \mapsto \hat K^{-1}(M)$. In \cite{TWZ3} the remaining data of a differential extension of  $K^{-1}$ are defined. To conclude this section, we review these here as they'll be used later for the construction of $S^1$-integration.

\begin{defn}[\cite{TWZ3}]\label{DEF:odd-diff-K-theory-R-a-maps}
We define $R = Ch : \hat K^{-1} (M) \to \Om^{\odd}_{cl} (M) $ and $I: \hat K^{-1} (M) \to K^{-1} (M)$ to be the forgetful map, sending a $CS$-equivalence class of a map $g:M\to U$ to its underlying homotopy class. This yields a commutative diagram 
\[
\xymatrix{
 & K^{-1}(M) \ar[rd]^{[Ch]} & \\
 \hat K^{-1}(M) \ar[ru]^I \ar [rd]^{R} & &H^{\odd}(M) \\
 & \Om^{\odd}_{cl} (M) \ar[ru]_-{\textrm{deRham}} &
}
\]

The remaining data is given by a map  $a: \Om^{\even}(M;\R) / Im(d) \to \hat K^{-1}(M)$, constructed as follows.
To define the map $a$ we first construct an isomorphism
 \[
 \widehat{CS} : Ker(I) \to \left( \Om^{\textrm{even} } (M) / Im(d) \right) / Im([Ch])
  \]
  where 
  \[
Ker(I) = \{ [g] | \, \textrm{ there is a path $g_t$ such that $g_1 = g$ and $g_0=1$}\}.
\]
 The map $\widehat{CS}$  is defined for $[g] \in  Ker(I) \subset \hat K^{-1}(M)$ by choosing a (non-unique) map $g_t: M \times I \to U$ such that $g_1 = g$ and $g_0=1$ is the constant map  $M \to \U$ to the identity of $\U$,  and letting
\[
\widehat{CS}([g]) = CS(g_t)   \quad  \in \left( \Om^{\textrm{even} } (M) / Im(d) \right) / Im([Ch]).
\]
 According to \cite{TWZ3}, this map is well defined independently of choice of representative since, modulo exact forms, every even degree $CS$-form on $M$ of a loop $M \to \Om(U)$ can be written as an even degree Chern form of some connection; see \cite[Theorem 3.5]{TWZ3}. Moreover this map is surjective since, modulo exact forms, every even form on $M$ is a the $CS$-form of some path; see \cite[Corollary 5.8]{TWZ3}.

Finally, the map $a= \widehat{CS}^{-1} \circ \pi$ is defined to be the composition of the projection $\pi$ with $\widehat{CS}^{-1}$,
 \[
\xymatrix{
\Om^{\textrm{even} } (M) /  Im(d) \ar[r]^-{\pi} & \left( \Om^{\textrm{even} } (M) /  Im(d) \right) / Im([Ch]) \ar[r]^-{\widehat{CS}^{-1}} & Ker(I) \subset \hat{K}^{-1}(M)
}
\]
and this map satisfies $Ch \circ a = d$, $Ker(a) = Im([Ch])$, and $Im(a) = Ker(I)$, yielding the exact sequence 
\[
\xymatrix{
K^{*-1}(M) \ar[r]^-{[Ch]} &  \Om^{*-1}(M;\R) / Im(d) \ar[r]^-{a} & \hat K^{*}(M) \ar[r]^{I} & K^*(M) \ar[r]^{\quad 0} & 0.
} 
\]
\end{defn}

\subsection{Model for $BU \times \Z$ and the even differential extension}\label{SEC:BUxZ-model}

We now recall the construction of the model of $BU\times \Z$  from McDuff \cite{McD}. We will use this model of $BU \times \Z$ to defined a differential extension $\hat K^0$ of $K^0$.

Let $U$ be the unitary group of operators on $\C_{-\infty}^{\infty}$, as in Definition \ref{defn;U}.
We denote by $I_k:\C_{-\infty}^{\infty}\to \C_{-\infty}^{\infty}$ the orthogonal projection onto $\C^k_{-\infty}$. Of particular interest to us will be the orthogonal projection $I_0$ onto $\C^0_{-\infty}$. 

\begin{defn}
A Hermitian operator on $\C_{p}^q$ is a linear map $h:\C_{p}^q\to \C_{p}^q$ such that $<h(x),y>=<x,h(y)>$. We denote by $H_{p}^q$ the space of Hermitian operators on $\C_{p}^q$ with eigenvalues in $[0,1]$. 
Note, that there are inclusions  $H_0^0\subset H_{-1}^1\subset H_{-2}^2\subset\dots$ given by
\[ h\in H_{-p}^p \quad \mapsto \quad 
Id_\C\oplus h\oplus 0 \in H_{-(p+1)}^{p+1}.
\]
Let  
\[
H=\bigcup_{p\geq 0} H_{-p}^p \label{DEF:H} 
\]
be the union of the spaces under the inclusions. In other words, $H$ is the set of Hermitian operators $h$ on $\C_{-\infty}^{\infty}$ with eigenvalues in the interval $[0,1]$, such that $h- I_0$ has finite rank. Again, we put the final topology on the space $H$; \emph{i.e.} $V\subset H$ is open iff $V\cap H_{-p}^p\subset H_{-p}^p$ is open for all $p$.
\end{defn}

There is an exponential map\label{definition-of-E}
\[
\exp:H\to U,\quad \exp(h):=e^{2\pi i h}=\sum_{n\geq 0}\frac {(2\pi i h)^n} {n!}.
\]
The fiber $\exp^{-1}(Id)$ of the identity $Id=Id_{\C_{-\infty}^\infty} \in U$ is the subset of $H$ of operators with eigenvalues in $\{0,1\}$,
\begin{multline*}
\exp^{-1}(Id) \\ = \{ P \in End( \C_{-\infty}^{\infty} ) | P \textrm{ is Hermitian}, spec(P) \subset \{0,1\}, \, \textrm{and } rank(P- I_0) < \infty \}.
\end{multline*}

The space $H$ is contractible with contracting homotopy $h(t)=th+(1-t)I_0$, providing a path which connects any $h \in H$ to $I_0 \in H$. Therefore, there is an induced map $E$ to the based loop space $\Om U$ of $U$. According to  \cite{D}, \cite{AP}, \cite{B} this map $E$ is a homotopy equivalence

\begin{prop}[\cite{McD}]\label{PROP:McDuff}
The map $E: \exp^{-1}(Id) \to \Om U$ given by
\[
E(P) (t)= e^{2\pi i \left( tP+ (1-t)I_0 \right) }
\]
is a homotopy equivalence.
\end{prop}

This justifies the following definition, which will be the model of $BU \times \Z$ used throughout this paper.

\begin{defn} \label{defnBU} 
We will denote by $BU\times \Z$ the space
\begin{multline*}
BU\times \Z:=\exp^{-1}(Id) \\ = \{ P \in End( \C_{-\infty}^{\infty} ) | P \textrm{ is Hermitian}, spec(P) \subset \{0,1\}, \, \textrm{and } rank(P- I_0) < \infty \}.
\end{multline*}
\end{defn}

Note, that there is a one-to-one correspondence between $BU \times \Z$ and the space of linear subspaces of $\C_{-\infty}^\infty$ which contain some $\C_{-\infty}^{p}$ and which are contained in some $\C_{-\infty}^{q}$,
\[
\exp^{-1}(Id) \cong \{ V \subset \C_{-\infty}^\infty\, |  \, \C_{-\infty}^{p} \subset V \subset \C_{-\infty}^{q} \, \textrm{for some $p,q \in \Z$ } \}.
\]
The equivalence is given by $P\mapsto V=Im(P)$ with inverse $V\mapsto proj_V$, where $proj_V$ denotes the orthogonal projection to a subspace $V\subset \C_{-\infty}^\infty$.

Next, we show how to recover the integer $\Z$ factor of $\B$. We define the $rank: \B \to \Z$ as follows. Let $P\in \B$ and let $V=Im(P)$ be the image of $P$, so that $\C_{-\infty}^{p} \subset V \subset \C_{-\infty}^{q}$ for some $p,q\in \Z$. With this, the rank of $P$ is defined to be
\begin{equation}\label{EQU;rank}
rank(P):= p+dim(V/\C_{-\infty}^p).
\end{equation}
Notice that the $rank(P)$ is a well-defined integer independent of the choice of $p$ and $q$.

The topology that $BU\times \Z$ inherits as a closed subspace of $H$ has the property that for any compact subset $K\subset BU\times \Z$, there exist integers $p$ and $q$ such that $K\subset (BU\times \Z)_p^q$, where
\begin{equation}\label{EQU:BUxZ-p-q}
(BU\times \Z)_p^q:=\{ P\in BU\times \Z \,|\,  \C_{-\infty}^{p}\subset Im(P)\subset \C_{-\infty}^{q}\}.
\end{equation}

Let $M$ be a  smooth compact manifold. A map $P: M \to BU\times \Z$ determines a vector bundle with connection over $M$, which is well defined up to the addition of the trivial line bundle with the trivial connection $d$, and a realization of this bundle as a sub-bundle of a trivial bundle, in the following way.  

\begin{defn}\label{DEF:P-assoc-bundle}
 Let  $P:M\to BU\times \Z$ be smooth. Choose $p, q \in \Z$ such that $\C_{-\infty}^{p}\subset Im(P)\subset \C_{-\infty}^{q}$, \emph{i.e.},
\[
Im(P(x))=\C_{-\infty}^{p}\oplus V(x), \text{ with }  V(x)\subset  \C_{p}^{q}
\]
for all $x \in M$. Let $E_P =\sqcup_{x\in M} V(x)$, which is a sub-bundle of the trivial bundle $M\times \C^{q}_{p}$. This vector bundle inherits a fiber-wise metric by restriction of the metric on $H$. The projection operator onto $V(x)$ defines a connection on $E_P$,  given by $\nabla_P (s) = P \circ d(s)$, which is compatible with the metric. Note that changing the  integer $p$, say by subtracting one, adds on the trivial line bundle with the trivial connection.  So we have an assignment
\[
P \mapsto \left( E_P , \nabla_P \right) 
\]
where the right hand side is well defined up to addition of $(\C^n, d)$, for some $n \in \N$.
\end{defn}

The curvature of the connection $\nabla_P$ is given by $R = P(dP)^2 = (dP)^2 P$. 
This justifies the following definition, in which we use the notion of plots and forms given by plots, where $BU \times \Z$ is filtered by the spaces $(BU\times \Z)_{-p}^p$ from equation \eqref{EQU:BUxZ-p-q}. In particular, a $k$-form $\alpha$ on $BU\times \Z$ is given by a sequence of forms  $\{\alpha_p\in\Om^k((BU\times \Z)_{-p}^p)\}_{p\geq 0}$ such that
\[  incl_p^*(\alpha_{p+1})=\alpha_{p}, \]
where  $incl_p:(BU\times \Z)_{-p}^p\subset (BU\times \Z)_{-(p+1)}^{p+1}$, $P\mapsto Id_\C\oplus P\oplus 0$, is the inclusion.

\begin{defn} \label{defnCh}
The universal even Chern form $Ch \in \Om^{\textrm{even}}_{cl}(BU\times \Z)$ is defined by 
\begin{equation}\label{EQ-Def:even-Ch}
Ch(P) := \Tr \sum_{n\geq 0} \frac{1}{(2 \pi i)^{n}} \frac{1}{n!} P(dP)^{2n}
\end{equation}
where, by definition, $\Tr(P) = rank(P)$, \emph{c.f.} equation \eqref{EQU;rank}.
\end{defn}

Note this is well defined since $PdP^2$ is invariant under pullback along the maps $incl_p$, since $d( Id_\C\oplus P\oplus 0) = 0 \oplus dP \oplus 0$.
We also have the associated Chern-Simons form.

\begin{defn} \label{defn;CS}
Denote by $P(\B)$ the path space of $\B$. The universal Chern-Simons form $CS \in \Om^{\textrm{odd}}(P(BU \times \Z) )$ is given using the evaluation map at time $t$, $ev_t:P(BU\times \Z)\to BU\times \Z$, by 
\begin{equation}\label{EQ-Def:odd-CS}
CS:= \int_{t \in I} ev_t^*(Ch).
\end{equation}
\end{defn}

It is straightforward to check that for a map $P_t: M \times I \to BU\times \Z$, \emph{i.e.} $P_t: M \to P(BU\times \Z)$, we have that the pullback $P_t^*(CS) \in \Om^{\odd}(M)$
is given by (\emph{c.f.} \cite[(2.2)]{TWZ3})
\begin{multline} \label{eq:evenCS}
CS(P_t) :=P_t^*(CS)\\
= \int_0^1 \sum_{n\geq 0} \frac{1}{(2 \pi i)^{n+1}} \frac{1}{n!}\Tr \left(  (Id-2P_t) \left( \frac{\partial}{\partial t}  P_t  \right) \overbrace{ dP_t \wedge \dots\wedge dP_t}^{2n+1 \text{ factors}} \right)dt 
\end{multline}
where $R_t =P_t(dP_t)^2$ is the curvature of $P_t$. 
Moreover, by Stokes' theorem we have that
\begin{equation}\label{EQ:odd-CS-Stokes}
d CS_\gamma = ev_1^*(Ch) - ev_0^*(Ch).
\end{equation}

Let $K^0(M) = [M,BU \times \Z]$ denote the homotopy classes of maps from $M$ to $BU\times \Z$. It follows from \eqref{EQ:odd-CS-Stokes} that there is an induced Chern homomorphism $[Ch]:K^{0}(M)\to H^{\even}(M)$, given by $P \mapsto [Ch(P)]$.

We now state some fundamental results concerning Chern and Chern-Simons forms  on a manifold.
\begin{thm}\label{THM:3-statements}
Let $M$ be a smooth compact manifold and consider the maps
\begin{align*}
Ch: Map^0(M, U) & \to \Om^{\odd}_{cl}(M;\R)  \\
Ch: Map^0(M, \B) & \to \Om^{\even}_{cl}(M;\R)  
\end{align*}
given by the pullback of forms \eqref{EQ-Def:odd-Ch} and \eqref{EQ-Def:even-Ch}, where $Map^0$ denotes the connected component of the constant map. Furthermore, consider the maps 
\begin{align*}
CS: Map_*(M \times I, U) &\to \Om^{\even}(M;\R)  \\
CS: Map_*(M \times I, \B) &\to \Om^{\odd}(M;\R)  
\end{align*}
given by the pullback of forms \eqref{EQ-Def:even-CS} and \eqref{EQ-Def:odd-CS}, where $Map_*$ indicates based maps whose time zero value is constant at $1 \in \U$, respectively $I_0 \in \B$. Finally, consider the induced maps 
\begin{align*}
CS_{\Om}: Map(M , \Om U) &\to \Om_{cl}^{\even}(M;\R)  \\
CS_{\Om}: Map(M , \Om (\B)) &\to \Om_{cl}^{\odd}(M;\R)  
\end{align*}
obtained by restricting the maps $CS$ above to those based paths that are based loops.
Then, for both even and odd degree cases, we have the following statements.
\begin{enumerate}
\item\label{statement-1} $Im(d) \subset Im(Ch)$, where $d$ is the DeRham differential.
\item\label{statement-2} $CS$ is onto.
\item\label{statement-3} $Im(Ch) \equiv Im(CS_{\Om})$ modulo exact forms.
\end{enumerate}
\end{thm}

\begin{proof}
The first statement \eqref{statement-1} for odd forms is given by Corollary 2.7 of \cite{TWZ3}. For positive degree even forms we know from Proposition 2.1 of \cite{PT} such exact even real forms can be written as the Chern form of connection on a trivial bundle, so the result follows from Narissman-Ramanan \cite{NR} and the fact that trivial bundles are represented by nulhomotopic maps.

The second statement \eqref{statement-2} for odd forms was proved in 
Corollary 2.2 of \cite{PT} (sharpening Proposition 2.10 of \cite{SS}),  and the second statement for even forms is proved as follows. Corollary 5.8 of \cite{TWZ3} states that the map $CS$ is onto $\Om^{\even}(M) / Im(d)$. By the previous case, we may write an exact error as an even Chern form, which by Theorem 3.5 of \cite{TWZ3}  may be written as the even $CS$-form of a based loop. By taking block sum of (or concatenating) the path with the loop, we obtain the desired result.

For the third statement \eqref{statement-3}, first note that the map $CS_{\Om}$ indeed lands in the space of closed forms, due to equations \eqref{EQ:even-CS-Stokes} and \eqref{EQ:odd-CS-Stokes}.

For the case of even forms, statement \eqref{statement-3} is Theorem 3.5 of \cite{TWZ3}. (Alternatively this also follows from Lemma \ref{lem;dbeta} below using the map $E$, by an argument very similar to the one that we give next for the case of odd forms.)  For the case of odd forms, we use Lemma \ref{lem;deta2} below, which shows that the map $h: \Om (\B) \to U$ induced by holonomy satisfies $CS_\Om(f)\equiv Ch(h\circ f) \in \Om^{\odd}_{cl}( M; \R )$ modulo exact forms. But the map $h$ is a homotopy equivalence, so by choosing any homotopy inverse we see that the other containment holds modulo exact forms as well.
\end{proof}

\begin{defn} \label{defn:CSequiveven}
Two maps $P_0, P_1 : M \to BU\times \Z$ are $CS$-equivalent if there is a smooth homotopy $P_t : M \times I \to BU\times \Z$  from $P_0$ to $P_1$ such that $P_t^*(CS) = CS(P_t) \in \Om^{\odd}(M)$ is exact.
\end{defn}

\begin{rmk}
It follows from Theorem \ref{THM:3-statements} that the previously introduced equivalence relation is the same as the 
equivalence relation: two maps $P_0, P_1 : M \to BU\times \Z$ are $CS$-equivalent if there is a smooth homotopy $P_t : M \times I \to BU\times \Z$  from $P_0$ to $P_1$ such that $P_t^*(CS) = CS(P_t) \in \Om^{\odd}(M)$ is an odd degree \emph{Chern form} on $M$.
\end{rmk}

We denote the set of $CS$-equivalence classes of maps from $M$ to $\B$ by the suggestive notation $\hat K^0 (M)$. This is a contravariant functor from the category of smooth manifolds and maps to sets. We wish to introduce an abelian group structure and show this is indeed a differential extension of $K$-theory. The naive definition of sum is given as follows, but suffers the same issue of not being continuous as does Definition \ref{Usum}, since it depends on non-canonical choices. Nevertheless, on plots it is well defined and associative. 

\begin{defn} \label{rmk;BUsum}
We define $\oplus: (\B) \times (\B) \to \B$ as follows. For $P,Q \in \B$ choose
maximal $p,m \in \Z$, and minimal $q,n \in \Z$, such that $Im(P)=\C_{-\infty}^{p}\oplus V$ and $Im(Q)=\C_{-\infty}^{m}\oplus W$,
where $V$ and $W$ are finite dimensional subspaces satisfying  $V\subset \C_{p}^q$ and $W\subset \C_{m}^n$. Then, define
\[
P \oplus Q = proj_{\C_{-\infty}^{p+m}\oplus s_m(V)\oplus s_{q}(W)}
\]
where $proj_Z$ denotes the orthogonal projection to a subspace $Z\subset \C_{-\infty}^\infty$, and $s_k$ denotes the operator given by  $s_k(e_i)=e_{i+k}$. Note that the above definition depends on the integers $p,q,m,n$, which, in general, may vary discontinuously.\end{defn}

We now  define a second sum operation, which is a continuous operation on $\B$, but is not associative. 

\begin{defn} \label{defn;BUsum}
We define $\boxplus: (\B)\times(\B)\to \B$ by defining a shuffle block sum $\boxplus:H\times H\to H$ and then by showing that this factors through $\B\subset H$. Recall that $H=H(\C^\Z_\cc)$ is the set of hermitian operators $h$ on $\C^\Z_\cc$ with eigenvalues in $[0,1]$ such that $h-I_0$ has finite rank, and, similarly, denote by $H(\C^\Z_\cc\oplus \C^\Z_\cc)$ the set of hermitian operators $\tilde h$ on $\C^\Z_\cc\oplus \C^\Z_\cc$ with eigenvalues in $[0,1]$, such that $\tilde h-(I_0\oplus I_0)$ has finite rank. In analogy with Definition \ref{Usum2}, we define the shuffle block sum to be the composition
\[
\boxplus: H(\C^\Z_\cc)\times H(\C^\Z_\cc)\stackrel{incl}{\hookrightarrow} H(\C^\Z_\cc\oplus \C^\Z_\cc)=H(\C^{\Z\sqcup\Z}_\cc)\stackrel{\tilde\rho_{sh}}{\to} H(\C^\Z_\cc),
\]
where the first map is the inclusion, and the second map is the map induced by the shuffle map $\rho_{sh}:\C_\cc^\Z\to \C_\cc^{\Z\sqcup \Z}$ from Definition \ref{Usum2} via $\tilde\rho_{sh}:H(\C^{\Z\sqcup\Z}_c)\to H(\C^\Z_c)$, $\tilde\rho_{sh}(\tilde h)= \rho_{sh}^{-1}\circ \tilde h\circ  \rho_{sh}$. Note that $\tilde\rho_{sh}$ indeed lands in $H(\C^\Z_\cc)$. Further, since the eigenvalues are preserved by $incl$ and $\tilde\rho_{sh}$, we see that $\B$ is preserved by $\boxplus$, inducing the desired map $\boxplus: (\B)\times(\B)\to \B$.
\end{defn}

\begin{rmk}\label{RMK:CS-additive}
With respect to the two operations we have $Ch(f \oplus g) = Ch(f) + Ch(g) = Ch(f \boxplus g) $ 
and $CS(f_t \oplus g_t) = CS(f_t) + CS(g_t) = CS(f_t \boxplus g_t) $, since in either case trace is additive. Moreover, for the path composition $g_t * h_t$, we have $CS(f_t * g_t) = CS(f_t) + CS(g_t)$.
\end{rmk}

We will first show in Lemma \ref{lem:CSzero} that restricted to any plot, the operation $\oplus$ is abelian up to a path that has vanishing $CS$-form, and also that each plot $P: M \to \B$ has an additive inverse $P^\perp: M \to \B$, again up to a path that has vanishing $CS$-form. Then, in Lemma \ref{thm:CSequivUsums}, we use similar techniques to show the operations $\oplus$ and $\boxplus$ are in fact homotopic, and even more, CS-equivalent, and so introduce the same operation on $\hat K^0(M)$. It will be useful for calculations and propositions that follows to have both operations. 
 
\begin{lem} \label{lem:CSzero}
For any $P, Q \in Map(M, BU \times \Z )$, there exists a path $\Gamma_t \in Map(M \times I, \B )$ such that 
\[ 
\Gamma_0 = P \oplus Q, \quad \quad \Gamma_1 = Q \oplus P, \quad \textrm{and} \quad  CS(\Gamma_t) =0.
\]
Also, for any $P \in Map(M, BU \times \Z)$, there exists a path $\Gamma_t \in Map(M \times I, \B  )$ such that 
\[ 
\Gamma_0 = P \oplus P^{\perp}, \quad \quad \Gamma_1 = I_0, \quad \textrm{and} \quad  CS(\Gamma_t) =0.
\]
Here $P^{\perp}$ is defined by 
\[
P^{\perp} = proj_{\C_{-\infty}^{-q} \oplus s_{-p-q} (V^\perp)}
\]
where $Im(P) = \C_{-\infty}^p \oplus V$ with $V \subset \C_p^q$, and $V^\perp\subset \C_p^q$ is the orthogonal complement of $V$ in $\C_p^q$, and $s_\ell$ denotes a ``shift'' by $\ell$ as in Definition \ref{rmk;BUsum}.
 \end{lem} 
 
 \begin{proof}
For the first statement, we assume that for some integers $m, n, r, s \in \Z$, the maps $P$ and $Q$ have images $Im(P) = \C_{-\infty}^{m} \oplus V$, where $V \subset \C_{m}^n$, and  $Im(Q)  = \C_{-\infty}^{r} \oplus W$, where $W \subset \C_{r}^s$. By picking $p>0$ large enough, we may assume without loss of generality, that $m=r=-p$ and $n=s=p$. We  construct a path of the form $\Gamma_t = id_{\C_{-\infty}^{-2p}} \oplus S_t \oplus 0|_{\C_{2p}^\infty}$, such that $\Gamma_t$ has vanishing $CS$-form and the path $\Gamma_t$ equals $P\oplus Q$ and $Q\oplus P$ at the two endpoints.

To simplify the notation we'll define our paths on the interval $[0,\pi/2]$, which can always be reparametrized to be a path on $I = [0,1]$, with the same properties as stated above.  We will use $X(t) :  \C_{-2p}^{2p} \to \C_{-2p}^{2p}$, defined by
\[
X(t) = 
\begin{bmatrix}
\cos t & \sin t  \\
-\sin t   & \cos t 
\end{bmatrix},
\]
where this block matrix acts on $ \C_{-2p}^{2p} =  \C_{-2p}^0 \oplus  \C_{0}^{2p}$. Furthermore, we denote by $A$ the $2p\times 2p$-matrix representing the map $P|_{\C_{-p}^p}:\C_{-p}^p\to \C_{-p}^p$, and by $B$ the $2p\times 2p$-matrix representing $Q|_{\C_{-p}^p}:\C_{-p}^p\to \C_{-p}^p$, and we use these to define the map $F:\C_{-2p}^{2p}\to \C_{-2p}^{2p}$,
\[ F = s_{-p}(P|_{\C_{-p}^p})\oplus s_{p}(Q|_{\C_{-p}^p})=\begin{bmatrix} A & 0 \\ 0 & B \end{bmatrix}. 
\]
With this notation, consider the path $S_t\in Map(M\times [0,\pi/2], End( \C_{-2p}^{2p}))$, 
\[
S_t  = 
X(t)FX(t)^{-1}
\]
so that $S_0 = \begin{bmatrix} A & 0 \\ 0 & B \end{bmatrix}$ and $S_{\pi/2} = \begin{bmatrix} B & 0 \\ 0 & A \end{bmatrix}$. Using the fact that $\frac{ \partial}{ \partial t} (X(t)^{-1}) = -X(t)^{-1}X'(t) X(t)^{-1}$, we get
\begin{eqnarray*}
S_t' &=&  X'(t)FX(t)^{-1} - X(t)FX(t)^{-1}X'(t) X(t)^{-1}
\end{eqnarray*}
and we have that
\[
 (dS_t)^{2n+1} =\left( X(t)dFX(t)^{-1} \right)^{2n+1} = X(t)(dF)^{2n+1}X(t)^{-1}
\]
Using the explicit formula for $CS(\Gamma_t)=CS(S_t)$ from \eqref{eq:evenCS}, it suffices to show
\[
Tr\Big( (Id-2S_t)  S_t'  (d S_t)^{2n+1}\Big)= 0.
 \]
 First, we have
\begin{multline*}
Tr\Big(  S_t'  (d S_t)^{2n+1}\Big)\\ = Tr \Big(  X(t)' F(dF)^{2n+1} X(t)^{-1} - X(t)FX(t)^{-1} X'(t)(dF)^{2n+1}X(t)^{-1} \Big)  \\
= Tr \Big(  X(t)^{-1} X(t)' F(dF)^{2n+1}  - X(t)^{-1} X'(t)F^{\perp}(dF)^{2n+1} \Big),
 \end{multline*}
 where we used  $F(dF)^2=(dF) F^\perp (dF) = (dF)^2 F$ and the fact that trace is cyclic. Next, we compute
 \begin{multline*}
 Tr\Big((S_t)  S_t'  (d S_t)^{2n+1}\Big)= Tr\Big( FX(t)^{-1} X'(t) F (dF)^{2n+1} - F X(t)^{-1} X'(t) (dF)^{2n+1}\Big) \\
 =  Tr\Big( X(t)^{-1} X'(t) (-F^\perp) (dF)^{2n+1} \Big)
  \end{multline*}
  since 
   \begin{multline*}
  Tr \Big( F X(t)^{-1} X'(t) F (dF)^{2n+1}  \Big) = Tr\Big( X(t)^{-1} X'(t) F F^\perp (dF)^{2n+1} \Big) =0. 
 \end{multline*}
 Putting this together and using $F - F^\perp -2( -F^\perp) = Id$ we have 
\begin{multline*}
Tr\Big( (Id-2S_t)  S_t'  (d S_t)^{2n+1}\Big)= Tr \Big( X(t)^{-1} X'(t) (dF)^{2n+1} \Big) = 0,
 \end{multline*}
 since $ X(t)^{-1} X'(t) = \begin{bmatrix} 0 & Id \\ -Id & 0 \end{bmatrix} $, and $(dF)^{2n+1} $ is block diagonal.

For the second statement, write $Im(P) = \C_{-\infty}^{p} \oplus V$, where $V \subset \C_{p}^q$. Denote by $V^\perp\subset \C_p^q$ the orthogonal projection of $V$ in $\C_p^q$, and let $P^\perp$ be the orthogonal projection onto $\C_{-\infty}^{-q}\oplus s_{-p-q}(V^\perp)$. By Definition \ref{rmk;BUsum}, $P \oplus P^\perp$ is the projection onto $\C_{-\infty}^{p-q} \oplus s_{-q}(V) \oplus s_{-p}(V^\perp)$. Notice that $s_{-q}(V) \subset \C_{p-q}^0$ and $s_{-p}(V^\perp) \subset \C_0^{q-p}$. Similarly to the previous case, let $X(t):\C^{q-p}_{p-q}\to \C^{q-p}_{p-q}$ be the map
\[ X(t) = \begin{bmatrix}\cos t & \sin t  \\-\sin t   & \cos t \end{bmatrix}\]
regarded as an endomorphism of $\C_{p-q}^0\oplus \C_0^{q-p}$. We denote by $A$ the $(q-p)\times (q-p)$-matrix representing the operation $P|_{\C_p^q}:\C_p^q\to \C_p^q$, and by $A^\perp$ the $(q-p)\times (q-p)$-matrix representing $P^\perp |_{\C_{-q}^{-p}}:\C_{-q}^{-p}\to \C_{-q}^{-p}$. Denote by $G:\C_{p-q}^{q-p}\to \C_{p-q}^{q-p}$ and $H:\C_{p-q}^{q-p}\to \C_{p-q}^{q-p}$ the maps
\[
G = s_{-q}(P|_{\C_p^q})\oplus Id|_{\C_0^{q-p}}= \begin{bmatrix}
A & 0 \\
0 & Id
\end{bmatrix},
\quad \quad
H = Id|_{\C_{p-q}^0}\oplus s_{q}(P^\perp |_{\C_{-q}^{-p}}) =\begin{bmatrix}
Id & 0 \\
0 & A^{\perp}
\end{bmatrix}.
\]
Consider the path $S_t\in Map(M\times [0,\pi/2], End( \C_{p-q}^{q-p}))$, 
\[
S_t  = X(t) G X(t)^{-1} H
=\begin{bmatrix}
A + \sin^2(t)A^\perp & \cos (t) \sin (t) A^\perp \\
\cos(t) \sin (t) A^\perp & \cos^2(t) A^\perp
\end{bmatrix}.
\]
Then $S_0 = \begin{bmatrix}
A & 0 \\
0 & A^{\perp}
\end{bmatrix}$, and $S_{\pi/2} =  \begin{bmatrix}
Id & 0 \\
0 & 0
\end{bmatrix}$.
Also, $S_t$ is symmetric and $S_t^2 = S_t$, for each $t$,
so $\Gamma_t  :=  id_{\C_{-\infty}^{-p}} \oplus S_t \oplus 0$ satisfies $\Gamma_t :  M\times [0,\pi/2]\to \B$, and  
$\Gamma_0 = P \oplus P^\perp$ and $\Gamma_1 = I_0$. We then calculate
\[
dS_t =
\begin{bmatrix}
\cos^2(t)  & -\cos (t) \sin (t)  \\
- \cos(t) \sin (t)  & -\cos^2(t)
\end{bmatrix}
dA,
\quad \quad
(dS_t)^2 =
\begin{bmatrix}
\cos^2(t)  &  0 \\
0  & \cos^2(t) 
\end{bmatrix}
(dA)^2,
\]
and
\[
\frac{\partial}{ \partial t} S_t =
\begin{bmatrix}
2 \cos (t) \sin (t) A^\perp  &  (\cos^2 (t)- \sin^2 (t)) A^\perp \\
 (\cos^2(t)-  \sin^2 (t)) A^\perp   & -2 \cos(t) \sin (t) A^\perp
\end{bmatrix},
\]
and with this,
\begin{multline*}
(Id-2S_t)  S_t'  (d S_t)^{2n+1}
=(Id-2S_t)  S_t'  dS_t\cdot (d S_t)^{2n}\\
=
\begin{bmatrix}
-\sin(t)\cos(t)\cdot A^\perp & -\cos^2(t)  \cdot A^\perp \\
-\cos^2(t) \cdot A^\perp & \sin(t)\cos(t)  \cdot A^\perp
\end{bmatrix} \cdot \cos^{2n}(t) (dA)^{2n+1}.
\end{multline*}
Thus, we see that we have a vanishing trace,
\[
Tr((Id-2S_t)  S_t'  (d S_t)^{2n+1})=0.
\]
It therefore follows from equation \eqref{eq:evenCS} that $CS(\Gamma_t) = CS(S_t)=0$, which is the claim.
\end{proof}

We now show the operations $\oplus$ and $\boxplus$ are CS-equivalent.

\begin{lem}\label{thm:CSequivUsums}
Let $M$ be a smooth compact manifold.
\begin{enumerate} \item\label{transpositions-CS-equiv}
Let $k\neq \ell\in\Z$ and let $\tilde\tau: \B \to \B$ be the map induced by the isomorphism $\tau:\C_{-\infty}^\infty\to \C_{-\infty}^\infty$ which exchanges the $k^{th}$ and the $\ell^{th}$ basis vectors of $\C_{-\infty}^\infty$.
Then there is a map $\Gamma_t:\B\times I\to \B$, \emph{i.e.} $\Gamma:\B\to P(\B)$, such that $\Gamma(P)_0=P$ and $\Gamma(P)_1=\tilde\tau\circ P$, and $CS(\Gamma(P)_t)=0$. \\ In particular, any plot $g:M\to \B$ is CS-equivalent to $\tau\circ g$.
\item\label{item:oplus-vs-boxplus}
Let $P,Q: M \to \B$. There is a smooth homotopy $\Gamma_t : M \times I \to \B$ satisfying $\Gamma_0 = P \oplus Q$, $\Gamma_1 = P \boxplus Q$ and $CS(\Gamma_t) = 0$. \\ In particular, $P \oplus Q$ is $CS$-equivalent to $P \boxplus Q$,  so that 
\[
[P \oplus Q] =[P \boxplus Q] \in \hat K^0(M)
\]
\end{enumerate}
\end{lem}
\begin{proof}

First the first statement we may assume $\ell = k+1$ since any permutation can obtained by a composition of such transposition and $CS$ is additive with respect to compositions, \emph{c.f.} Remark \ref{RMK:CS-additive}. For this case we use the rotation family $X(t)$ given by $\begin{bmatrix} \cos(t)& \sin(t)\\ -\sin(t) & \cos(t) \end{bmatrix}$ on $\C_k^{k+2}$ and by the identity everywhere else (compare this with the proof of the first part of the previous Lemma \ref{lem:CSzero}). Setting $\Gamma(P)_t=X(t) P X(t)^{-1}$, the same argument and calculation as in the first part of the previous Lemma \ref{lem:CSzero} shows that $\Gamma(P)_t$ swaps the $k^{th}$ and $(k+1)^{st}$ basis vectors and has a vanishing $CS$-form.

Now, for the second claim, for fixed plots $P,Q$, we may find numbers $p, q, m, n$, such that $Im(P(x))=\C_{-\infty}^{p}\oplus V(x)$ with  $V(x)\subset \C_{p}^q$ for all $x\in M$, and $Im(Q(x))=\C_{-\infty}^{m}\oplus W(x)$ with $W(x)\subset \C_{m}^n$ for all $x\in M$. Without loss of generality, we may assume that $p=m$ (by taking the minimum of $p$ and $m$) and $q=n$ (by taking the maximum of $q$ and $n$). Then $P(x) \oplus Q(x) = proj_{\C_{-\infty}^{2p}}+proj_{s_p(V(x))}+proj_{s_{q}(W(x))}$ has an image $Im(P(x) \oplus Q(x))=\C_{-\infty}^{2p}\oplus s_p(V(x))\oplus s_q(W(x))$. Thus, the direct sum $\oplus$ (for these plots) can be expressed as the composition
\[
\quad\quad\oplus: H(\C^\Z_\cc)\times H(\C^\Z_\cc)\stackrel{incl}{\hookrightarrow} H(\C^\Z_\cc\oplus \C^\Z_\cc)=H(\C^{\Z\sqcup\Z}_\cc)\stackrel{\tilde\rho}{\to} H(\C^\Z_\cc),
\]
where $\rho:\C_\cc^\Z\to \C_\cc^{\Z\sqcup\Z}$ is given by
\[
\hspace{1.3cm}\rho(e_k)=\left\{
\begin{array}{ll}
 \rho_{sh}(e_k), & \text{for }k<2p, \\
 e_{k-p}, \text{ in the first $\Z$ component} & \text{for } 2p\leq k< p+q, \\
 e'_{k-(p+q)}, \text{ in the second $\Z$ component} & \text{for } p+q\leq k< 2q, \\ 
 \rho_{sh}(e_k), & \text{for }  2q\leq k.
\end{array}\right.
\]
Therefore, $P(x)\oplus Q(x)$ and $P(x)\boxplus Q(x)\in H(\C^\Z_\cc)$ differ only  on $\C_{2p}^{2q}$, and there $P(x)\oplus Q(x)=\tilde\alpha (P(x)\boxplus Q(x))$ for some isomorphism $\alpha:\C_{2p}^{2q}\to \C_{2p}^{2q}$ which relabels the basis elements of $\C_{2p}^{2q}$. Each such isomorphism is a composition of transpositions as in part \eqref{transpositions-CS-equiv}, so that the result follows from the claim in \eqref{transpositions-CS-equiv}.
\end{proof}

\begin{cor}
The set of $CS$-equivalence classes of maps of $M$ into $BU \times \Z$ has an abelian group structure induced by $\oplus$ or $\boxplus$, with identity $[I_0]$, and so $M \to \hat K^0(M)$ defines a contravariant functor from compact manifolds to 
abelian groups.
\end{cor}

We now show that the functor $\hat K^0$ admits the data of a differential extension.
By Theorem \ref{thm:uniqueness-of-diff-k-theory}, \emph{i.e.} \cite[Theorem 3.3]{BS3} of Bunke and Schick, it then follows that this model is isomorphic to any other model of even differential $K$-theory, via a unique natural isomorphism. 

\begin{defn}\label{DEF:even-diff-K-theory-R-a-maps}
Let $I:\hat K^0(M)\to K^0(M)$ denote the forgetful map which sends a $CS$-equivalence class of maps to its homotopy class. By equation \eqref{EQ:odd-CS-Stokes} we have a well defined map $R=Ch:\hat K^0(M) \to \Om^{\even}_{cl}(M)$, and by definition of the Chern character \eqref{EQ-Def:even-Ch} we have a commutative diagram
\[
\xymatrix{
 & K^{0}(M) \ar[rd]^{[Ch]} & \\
 \hat K^{0}(M) \ar[ru]^I \ar [rd]^{Ch} & &H^{\even}(M) \\
 & \Om^{\even}_{cl} (M) \ar[ru]_-{\textrm{deRham}} &
}
\]

The remaining data is given by a map  $a: \Om^{\odd}(M;\R) / Im(d) \to \hat K^{0}(M)$, constructed as using the ideas from \cite{SS}. To define the map $a$ we first construct an isomorphism
 \[
 \widehat{CS} : Ker(I) \to \left( \Om^{\textrm{odd} } (M) / Im(d) \right) / Im([Ch])
  \]
  where 
  \begin{multline*}
Ker(I) = \{ [P] | \, \textrm{ there is a path $P_t:M\times I \to \B$ such that $P_1 = P$} \\ \text{and $P_0= I_0:M\to \B$ is the constant map to $I_0$}\}.
\end{multline*}
 The map $\widehat{CS}$  is defined for $[P] \in  Ker(I) \subset \hat K^{0}(M)$ by choosing a (non-unique) $P_t: M \times I \to U$ with $P_1 = P$ and $P_0=I_0$, and letting
\[
\widehat{CS}([P]) = CS(P_t)   \quad  \in \left( \Om^\odd (M) / Im(d) \right) / Im([Ch]).
\]
We first show this map is well defined. For two different choices $P_t$ and $Q_t$ satisfying $P_1 = Q_1 = P$ and $P_0 = Q_0 = I_0$,  consider the composition $F_t = Q_{-t} * P_t : M \to \Om (\B)$. But $CS(P_t)-CS(Q_t)=CS(F_t)\in Im(Ch)$ mod exact forms by Theorem \ref{THM:3-statements}\eqref{statement-3}, and $Im(d)\subset Im(Ch)$ by Theorem \ref{THM:3-statements}\eqref{statement-1}. Thus, $CS(P_t)$ and $CS(Q_t)$ coincide modulo $Im(Ch)$, which shows that $CS$ is well defined. Next, the map $\widehat{CS}$ is onto, since $CS$ is onto by Theorem \ref{THM:3-statements}\eqref{statement-2}. 

Finally, we show the map $\widehat{CS}$ is one to one. Suppose $[P] \in Ker(I)$, and that  so that $\widehat{CS}([P])=0$, \emph{i.e.} that $CS(P_t) \in Im(Ch)$ for some choice of $P_t$ as above. Then $CS(P_t) = Ch(g)$ for some $g: M \to U$. 
By Theorem \ref{THM:3-statements}\eqref{statement-3} we have $Ch(g) = CS(Q_t)$,  modulo exact, for some
$Q_t : M \to \Om B U \times \Z$. Let $K_t = Q_{1-t} * P_t$ be the concatenation, which is a new homotopy from $P$ to the constant $I_0$, and satisfies $CS(K_t) = CS(P_t) - CS(Q_t) = Ch(g)- Ch(g) = 0$ modulo exact. Thus, $K_t$ is a $CS$-exact path from $P$ to $I_0$, so that $[P]=[I_0]=0$.

Using the map $\widehat{CS}$, we define the map $a$ as the composition of the projection $\pi$ with $\widehat{CS}^{-1}$,
 \[
\xymatrix{
a:\Om^{\textrm{odd} } (M) /  Im(d) \ar[r]^-{\pi} & \left( \Om^{\textrm{odd} } (M) /  Im(d) \right) / Im([Ch]) \ar[r]^-{\widehat{CS}^{-1}} & Ker(I) \subset \hat{K}^{0}(M)
}
\]
yielding the exact sequence 
\[
\xymatrix{
K^{-1}(M) \ar[r]^-{[Ch]} &  \Om^{\odd}(M;\R) / Im(d) \ar[r]^-{a} & \hat K^{0}(M) \ar[r]^{I} & K^0(M) \ar[r]^{\quad 0} & 0.
} 
\]

It remains to show this map satisfies $Ch \circ a = d$. This equation follows, since $d\circ \widehat{CS}=Ch$ as maps $Ker(I)\to \Om^{\even}_{cl}(M)$, which can be seen calculating for any $[P] \in Ker(I)$ and $P_t$ connecting $P_1=P$ with $P_0=I_0$,
\[
d (\widehat{CS}([P]))=d(CS(P_t))  \stackrel{\eqref{EQ:odd-CS-Stokes}}{=} Ch(P_1) - Ch(P_0) = Ch(P_1)=Ch(P)=Ch([P]).
\]
\end{defn}

\begin{rmk}
Note that the uniqueness of even differential $K$-theory does not require an $S^1$-integration, see \cite[Theorem 3.3]{BS3}. Since the definition of $\hat K^0(M)$ given here satisfies the axioms in Definition \ref{defn:diffext}, this is already isomorphic to any other model of even differential $K$-theory. For example, $\hat K^0(M)$ coincides with Simons and Sullivan's model $\hat K^0_{SS} (M)$ in \cite{SS}, given by the Grothendieck group of structured vector bundles (i.e. isomorphism classes of $CS$-equivalence classes of bundles with connection). There is a natural map between these two models, given by the ``pullback'': for $P : M \to BU \times \Z$ consider the following difference of structured  vector bundles
\[
P^* = (E_P, \nabla_P) - (\C^k, d)
\]
where the choice $(E_P, \nabla_P)$ is made as in Definition \ref{DEF:P-assoc-bundle}, and  
$k =  rank(E_P) - rank(P)$. Note this is well defined since $(E_P, \nabla_P)$ is well defined up to adding a trivial line bundle with trivial connection. In fact, this map induces an isomorphism, as we now explain.

Given a smooth map $f_t: M \times I \to BU \times Z$ such that $CS(f_t)$ is exact, consider a choice for the total pullback bundle $P^* = (E, \nabla) - (\C^k, d)$ over  $M \times I$ whose slice at time $t \in I$ is $(E_t, \nabla_t) - (\C^k, d)$.
Parallel transport $P_t: E_0 \to E_t$ provides an isomorphism between $(E, \nabla)$ and the
product bundle $E_0 \times I$, and a path of connections $P_t^* \nabla_t$ on $E_0$ such that 
 $CS(P_t^* \nabla_t) = CS(f_t)$ is exact. So, $(E_0, \nabla_0)$ and $(E_0, P_1^* \nabla_1) \cong (E_1,\nabla_1)$ are $CS$-equivalent, and the integer $k$ is time independent, which shows the pullback map is well defined.

Since any element of $\hat K^0_{SS} (M)$ has a representative of the form $\left( (E, \nabla), (\C^k, d) \right)$ for some $k$ (c.f. \cite{SS} (3.1)), the map is a surjective  by the the Narasimhan-Ramanan theorem \cite{NR}, and a group homomorphism by Lemma \ref{thm:CSequivUsums}.

Finally, the map is one to one, as we now sketch. 
Given two maps $f_0, f_1: M \to BU \times \Z$ whose pullbacks are equal in $\hat K^0_{SS}(M)$, we may as well  assume the second (trivial) summands
$\C^k$ of the pullbacks are equal, and even that they are zero, since $f_0 \cong f_0 \boxplus \C^k \boxplus (\C^k)^\perp$, and similarly for $f_1$. 
Then it suffices to consider $(E_0,\nabla_0)$ and $(E_1, \nabla_1)$ that represent the same isomorphism class of CS-equivalence class of bundles with connection. Using the isomorphism to transport both connections to $E_0$,  we have two connections, say $\nabla_0$ and $\nabla_1$ on $E_0$, and a path of connections $\nabla_t$ from  $\nabla_0$ to $\nabla_1$ such that $CS(\nabla_t)$ is exact. We can regard this as a bundle with connection on $M \times I$, and represent this as a map 
$G_t: M \times I \to BU \times \Z$ using the Narasimhan Ramanan theorem \cite{NR}, so that $CS(G_t)$ is exact, by assumption.
The pullback bundles via $G_0$ and $G_1$ may not equal $E_0$ and $E_1$ with their connections (respectively), but
since they are isomorphic, there are homotopies $F_t: M \times I \to BU \times \Z$ from $f_0$ to $G_0$,
and $H_t: M \times I \to BU \times \Z$ from $G_1$ to $f_1$ defined explicitly by families of rotations of coordinates, by the main theorem of \cite{Sc}. By an argument similar to Lemma \ref{lem:CSzero} above, for such homtopies given by rotations we have $CS(F_t) = CS( H_t) =0$. Therefore $F_t * G_t * H_t$ is a homotopy from $ E_0$ to $E_1$
and
\[
CS(F_t * G_t * H_t) = CS(F_t) +  CS(G_t) + CS( H_t) = CS(G_t)
\]
is exact, as desired. The details of this argument will be exploited in \cite{TWZ4}, to provide a means for classifying structured vector bundles.
\end{rmk}

\section{Constructing $S^1$-Integration}\label{SEC:S1-integration}

Our goal in this section is to construct an $S^1$-integration map (Definition \ref{defn:S^1int}) for the differential extension of $K$-theory defined in Section \ref{SEC:Diff-K-Theory}. By Theorem \ref{thm:uniqueness-of-diff-k-theory} (\emph{i.e.} \cite[Theorem 3.3]{BS3}) this  shows we have defined a model for differential $K$-theory. In Section \ref{SEC:odd-to-even} we discuss the even to odd part, while in Section \ref{SEC:even-to-odd} we discuss the odd to even part of the $S^1$ integration map.

Constructing the $S^1$ integration in these models is quite non-trivial, so we begin with a motivating discussion, focusing for concreteness on the integration map from the even to the odd part. Suppose that, for all compact manifolds with corners $M$, we had a natural map 
\[
i: Map( M \times S^1, BU \times \Z) \to Map (M, U)
\]
 that induces  the $S^1$-integration map 
$\int_{S^1} : K^{0}(M \times S^1) \to K^{-1}(M)$ in $K$-theory and makes the following diagram commute
\[
\xymatrix{
 Map( M \times S^1, BU \times \Z) \ar[r]^-{Ch}  \ar[d]^{i}&  \Om^{\textrm{even}}(M \times S^1)  \ar[d]^{\int_{S^1}}  \\
 Map(M,U)   \ar[r]^{Ch} &    \Om^{\textrm{odd}}(M).
 }
\]
By applying this diagram to the manifold $M \times I$ and integrating out the interval $I$ we obtain the commutative diagram
\[
\xymatrix{
 Map( M \times I \times  S^1, BU \times \Z) \ar[r]^-{Ch}  \ar[d]^{i}&  \Om^{\textrm{even}}(M \times I \times S^1)  \ar[d]^{\int_{S^1}}
   \ar[r]^{\int_{I}}
 &  \Om^{\textrm{odd}}(M \times S^1)  \ar[d]^{\int_{S^1}} \\
 Map(M \times I ,U)   \ar[r]^{Ch} &    \Om^{\textrm{odd}}(M \times I )  \ar[r]^{\int_{I}} & \Om^{\textrm{even}}(M) .
 }
\]
The composition of the top and bottom horizontal rows define the even and odd Chern-Simons forms, respectively.
The integration over $S^1$ is a chain map, since $S^1$ is closed.  So, if the Chern-Simons form for an element of 
$Map( M \times I \times  S^1, BU \times \Z)$ is exact, then so is the Chern-Simons form for the element
of $Map(M \times I ,U)$ induced by $i$. This shows there is an induced map $i : \hat K^{0}(M \times S^1) \to \hat K^{-1}(M)$ making the following diagram 
commute
\[
\xymatrix{
 \hat K^{0}(M \times S^1) \ar[r]^-{Ch}  \ar[d]^{i}&  \Om^{\textrm{even}}(M \times S^1)  \ar[d]^{\int_{S^1}}  \\
  \hat K^{-1}(M)   \ar[r]^{Ch} &    \Om^{\textrm{odd}}(M).
 }
\]

The maps $a: \Om^*(M) / Im(d) \to \hat K^{*+1}(M)$ induce maps $\Om^*(M) / Im(Ch) \to Ker(I) \subset \hat K^{*+1}(M)$ which are isomorphisms. In fact, in both the even and odd parts of the differential extensions given here, the map $a$ is defined to be the inverse of the map $\widehat{CS}:  Ker(I)  \to \Om^*(M) / Im(Ch) $  which is simply given by taking the Chern-Simons form. 
By the remarks above, the map induced by $i$ commutes with the Chern-Simons forms, so it follows that the diagram 
\[
\xymatrix{
\Om^{*+1}(M \times S^1) \ar[r]^{a} \ar[d]^{\int_{S^1}} & \hat K^{*}(M \times S^1) \ar[d]^{i}\\
\Om^{*}(M) \ar[r]^{a} & \hat K^{*+1}(M ) 
}
\]
commutes, and so we would have produced an $S^1$-integration map according to Definition \ref{defn:S^1int}.

Unfortunately, it is difficult to give natural maps which commute on the nose with these geometric representative for the Chern Character. The natural candidates for these two maps from Bott periodicity and homotopy theory do not make the diagrams commute on the nose, but rather commute only up to an exact differential form, compare \cite{FL}.
This exact error makes the previous argument fail at the first step, since integration over the interval is not a chain map.

The resolution is to study this exact error and use it to define a new integration map. We will show using methods of homotopies and associated transgression forms how to  correct such a situation, in the end yielding a bona fide $S^1$-integration map.

\subsection{Odd to Even}\label{SEC:odd-to-even}
Recall from section \ref{SEC:BUxZ-model} that $H$ is the subspace of hermitian operators $h$ on $\C_{-\infty}^{\infty}$ with eigenvalues in $[0,1]$ such that $h-I_0$ has finite rank (see page \pageref{DEF:H}). Recall furthermore from page \pageref{definition-of-E}, that we have the exponential map $\exp:H\to U$, $\exp(h)=e^{2\pi i h}=\sum_{n\geq 0}\frac {(2\pi i h)^n} {n!}$, and from Proposition \ref{PROP:McDuff}, that we denote by $E:BU\times\Z\to \Om U$ the exponential of the straight line to $I_0$, 
\[
E(P) :S^1\to U \quad \quad \textrm{given by} \quad \quad E(P)(t)=e^{2 \pi i (tP +(1-t)I_0)}.
\]
Recall from \eqref{EQ-Def:even-Ch} that $Ch\in \Om^{\even}(\B)$, and from \eqref{EQ-Def:even-CS} that $CS\in \Om^{\even}(PU)$, which we may restrict from the path space on $U$ to the based loop space of $U$, also denoted by $CS\in \Om^{\even}(\Om U)$.
The next lemma shows that $E^*CS$ equals $Ch$ on $\B$ modulo exact.
\begin{lem} \label{lem;dbeta}
There exists a form $\beta\in \Omega^{\odd}(BU\times \Z)$ such that
\begin{equation}\label{EQU:dbeta}
 d\beta=\E^*CS-Ch. 
\end{equation}
In other words, the following diagram commutes modulo $d\beta$ for all compact manifolds with corners.
\[
\xymatrix{
 Map(M \times S^1,U) \ar[r]^{Ch}  &  \Om^{\textrm{odd}}(M \times S^1) \ar[d]^{\int_{S^1}} \\
 Map(M, BU \times \Z)  \ar[u]^{\E^\sharp}  \ar[r]^-{Ch} &    \Om^{\textrm{even}}(M)
}
\]
\end{lem}
\begin{proof}
For $k\in \Z$, consider the maps $\gamma_k:BU\times \Z\to Map([0,1],H)$ given by taking the straight line path $\gamma_k(P)(t)=tP+(1-t)I_k$ from $I_k$ to $P$, where $I_k$ is the projection to $\C_{-\infty}^k$ as before. Also denote by $\rho_k:BU\times \Z\to Map([0,1],H)$ the constant map to the straight line path $\rho_k(P)(t)=tI_k+(1-t)I_{k-1}$ from $I_{k-1}$ to $I_{k}$. Furthermore, there is a map $h_k:BU\times \Z\to Map([0,1]\times [0,1],H)$  given by $h_k(P)(s,t)=tP+(1-t)(sI_k+(1-s)I_{k-1})$ such that the boundary consists of $\partial h_k(P)=\gamma_{k}(P)-\gamma_{k-1}(P)-const_P+\rho_k(P)$. Note that $\gamma_k$, $\rho_k$, and $h_k$ can be composed with $\exp:H\to U$, giving rise (by slight abuse of notation) to maps $\exp\circ \gamma_k:BU\times \Z\to Map([0,1],U)$, etc.

According to \cite[Propositions 3.4 and 3.2]{TWZ3}, one can construct a form $H\in \Omega^{\odd}(Map([0,1]\times [0,1],U))$ such that $dH=\partial_1^* CS-\partial_2^* CS-\partial_3^* CS+\partial_4^* CS$, where $\partial_i:Map([0,1]\times [0,1],U)\to Map([0,1],U)$ is induced by the $i$th boundary component of $[0,1]\times [0,1]$, and $CS\in \Omega^{\even}(Map([0,1],U))$ is the universal Chern-Simons form. The form $H$ is determined by its pullback under any map $g:M\to Maps([0,1]\times[0,1], U)$ to a manifold with corners $M$,  and given by 
\begin{equation}\label{EQU:g*(H)}
g^*(H)=\Tr\sum_{n\geq 1}\sum_{i\neq j} c_{n,i,j} \int_{0}^1 \int_0^1 \overbrace{(g^{-1}dg) \dots \underbrace{(g^{-1} \frac{\partial}{\partial t}g)}_{i^{th}}\dots \underbrace{(g^{-1}\frac{\partial}{\partial s}g)}_{j^{th}}\dots (g^{-1} dg)}^{2n+1\text{ terms}} dt ds,
\end{equation}
where $c_{n,i,j}$ are some constants. Note, that in the above situation, we have that 
\begin{eqnarray*}
(\exp\circ h_k)^*(dH) &=&(\exp\circ h_k)^*(\partial_1^*CS-\partial_2^*CS-\partial_3^*CS+\partial_4^*CS) \\
&=&(\exp\circ \gamma_k)^*CS-(\exp\circ\gamma_{k-1})^*CS+(\exp\circ \rho_k)^*CS. 
\end{eqnarray*}

We then define $\beta\in \Omega^{\odd}(BU\times \Z)$ by setting
\[ \beta=\sum_{k\leq 0} (\exp\circ h_k)^*H=(\exp\circ h_0)^*H+(\exp\circ h_{-1})^*H+(\exp\circ h_{-2})^*H+\dots \]
We claim that $\beta$ is a well defined odd form on $BU\times \Z$, that is, the infinite sum reduces to a finite sum whenever applied to tangent vectors $v_1,\dots v_\ell$ at some $P\in BU\times \Z$. Indeed, if we represent $v_1,\dots, v_\ell$ by a map $f:B(0)\to BU\times \Z$ from a compact ball $B(0)\subset \R^\ell$ centered at $0$ with $f(0)=P$ and $f_*(\frac{\partial}{\partial x^j})=v_j$, then the image of $f$ is contained in some subspace $\{P\in BU\times \Z\,|\, P|_{\C_{\infty}^r}=Id_{\C_{\infty}^r}\}$ for some $r$ (see page \pageref{EQU:BUxZ-p-q}). Now, for any $k<r$, we claim that 
\begin{equation}\label{EQU:hkH-finiteness}
(\exp\circ h_k\circ f)^*H=0,
\end{equation}
so that $\beta(v_1,\dots, v_\ell)=f^*\left(\sum_{k\leq 0} (\exp\circ h_k)^*H\right)(\frac{\partial}{\partial x^1},\dots, \frac{\partial}{\partial x^\ell})$ becomes a finite sum. To see \eqref{EQU:hkH-finiteness}, let $g=\exp\circ h_k\circ f:B(0)\to Map([0,1]\times [0,1],U)$, which is given by 
\begin{eqnarray*}
g(x)(s,t)&=&e^{2\pi i(tf(x)+(1-t)(sI_k+(1-s)I_{k-1}))}\\
&=&e^{2\pi itf(x)}\cdot e^{2\pi i(1-t)(sI_k+(1-s)I_{k-1})},
\end{eqnarray*}
where the last equality follows, since $f(x)$ commutes with $(sI_k+(1-s)I_{k-1})$ for $k<r$, as $f$ is the identity on $\C_{-\infty}^r$. If we decompose $\C_{-\infty}^\infty=\C_{-\infty}^k\oplus \C_{k+1}^\infty$, then we can see that $g(x)(s,t)$ preserves these subspaces for each $x\in B(0)$ and $s,t\in [0,1]$; that is, $g(x)(s,t)$ maps $\C_{-\infty}^k$ to $\C_{-\infty}^k$ and $\C_{k+1}^\infty$ to $\C_{k+1}^\infty$. Thus, $g^{-1}$, $dg$, $\frac{\partial}{\partial t}g$, and $\frac{\partial}{\partial s}g$ also preserve this decomposition. Finally note, that $g^{-1}\frac{\partial}{\partial s}g=2\pi i (1-t)(I_k-I_{k-1})$ which vanishes on $\C_{k+1}^\infty$, while $g^{-1}dg=e^{-2\pi itf(x)}d(e^{2\pi itf(x)})$ which vanishes on $\C_{-\infty}^k$. Thus, by \eqref{EQU:g*(H)}, we see that $g^*(H)=0$, which is the claim of \eqref{EQU:hkH-finiteness}.

It remains to check equation \eqref{EQU:dbeta}. We calculate $d\beta$ by evaluating the first $k$ terms in the expansion of $\beta$ as follows,
\begin{eqnarray*}
d\beta 
&=& d((\exp\circ h_0)^*H+(\exp\circ h_{-1})^*H+(\exp\circ h_{-2})^*H+\dots)\\
&=& (\exp\circ h_0)^*dH+(\exp\circ h_{-1})^*dH+(\exp\circ h_{-2})^*dH+\dots\\
&=& (\exp\circ \gamma_0)^*CS-(\exp\circ\gamma_{-1})^*CS+(\exp\circ \rho_0)^*CS\\
&&+(\exp\circ \gamma_{-1})^*CS-(\exp\circ\gamma_{-2})^*CS+(\exp\circ \rho_{-1})^*CS\\
&&+(\exp\circ \gamma_{-2})^*CS-(\exp\circ\gamma_{-3})^*CS+(\exp\circ \rho_{-2})^*CS+\dots\\
&&+(\exp\circ \gamma_{-(k-1)})^*CS-(\exp\circ\gamma_{-k})^*CS+(\exp\circ \rho_{-(k-1)})^*CS\\
&& +\sum_{j\geq k} (\exp\circ h_{-j})^*dH.
\end{eqnarray*}
Now, since $(\exp\circ \rho_j)(P)(t)=e^{2\pi i (tI_j+(1-t)I_{j-1})}$ is a constant map (independent of $P$) with winding number $1$, we see that $(\exp\circ\rho_j)^* CS=1$. Next, note that $\exp\circ \gamma_0(P)(t)=e^{2 \pi i (tP +(1-t)I_0)}=E(P)(t)$ so that the first term is $(\exp\circ \gamma_0)^*CS=E^*CS$, while the terms $(\exp\circ \gamma_{j})^*CS$ for $j<0$ cancel pairwise. Thus, for any $k\geq 0$, we have
\begin{equation*}
d\beta = E^*CS-(\exp\circ\gamma_{-k})^*CS+k +\sum_{j\geq k} (\exp\circ h_{-j})^*dH.
\end{equation*}
Now, the calculation from \cite[Theorem 3.5]{TWZ3} shows that for $k\geq 0$ sufficiently large (\emph{i.e.} when $P$ and $I_{-k}$ commute), we have for all $j\geq k$ that $(\exp\circ\gamma_{-j})^* CS = Ch+j$, where $j$ is a constant function. Thus, the terms $(\exp\circ h_{-j})^*dH=(\exp\circ \gamma_{-j})^*CS-(\exp\circ\gamma_{-(j+1)})^*CS+(\exp\circ \rho_{-j})^*CS$ vanish for all $j\geq k$. Using this together with $-(\exp\circ\gamma_{-k})^*CS+k=-Ch$, we obtain that $d\beta=E^*CS-Ch$, which is equation \eqref{EQU:dbeta}.
\end{proof}

We recall from \cite[Lemma 3.6]{TWZ3} that we may associate to each $g:M\to U^n_{-n}$ a map $\gamma_g:M\times I\to U^{2n}_{-2n}$, such that $\gamma_g(0)=g\oplus g^{-1}$ and $\gamma_g(1)=id$, and $CS(\gamma_g)=0$. The path $\gamma_g$ is essentially given as in Lemma \ref{lem:CSzero} using a $\sin$/$\cos$-matrix, but with $U$ instead of $\B$, but we repeat it here for completeness. For $g:M \to U^n_{-n}$ let $\gamma_g(t): M \times I \to U^{2n}_{-2n}$ be given by 
\[ \gamma_g(t) =  G X(t) H X(t)^{-1}, \] 
where (\emph{cf.} Lemma \ref{lem:CSzero})
\[
G = \begin{bmatrix}
g & 0 \\
0 & 1
\end{bmatrix}
\quad \quad
H = \begin{bmatrix}
1 & 0 \\
0 & g^{-1}
\end{bmatrix}
\quad \quad 
X(t) = 
\begin{bmatrix}
\cos (\pi t /2 ) & \sin (\pi t /2 ) \\
-\sin (\pi t /2 )  & \cos (\pi t /2 ) 
\end{bmatrix}
\]
so that $\gamma_g(0) = g \oplus g^{-1}$ and $\gamma_g(1) = id$\label{path:gamma_g}.
It is straightforward to check that for all $g,h:M \to U^n_{-n}$ mapping into the same components $U^n_{-n}$ of $U$, we obtain:
\begin{equation}\label{EQU:gamma-square-sum}
\gamma_{(g \boxplus h)} = \gamma_g \boxplus \gamma_h,
\end{equation}
where $g\boxplus h$ is given by a map $g\boxplus h:M\to U^{2n}_{-2n}$.

The obvious map from 
from the free loopspace $LU$ to the based loopspace $\Om U$, given by left $U$-action at the base point, does not preserve the $CS$ form. The following map, defined using $\boxplus$, does.

\begin{defn} \label{def;basedgt}
For any $g_t: M \to LU^n_{-n}\subset LU$ we define ${}_*g_t : M \to \Om U^{2n}_{-2n}\subset \Omega U$ by 
conjugating $g_t \oplus g_0^{-1}$ (under path composition ``$*$'') with the path $\gamma_{g_0}(t)$, i.e.
\[
{}_*g_t =  \gamma_{g_0}(t) * (g_t \oplus g_0^{-1})  * \overline{\gamma_{g_0}(t)},
\]
where $\overline\gamma$ denotes reversal of a path $\gamma$. The notation ${}_*g_t$ is meant to denote the \emph{based map} associated to the free loop $g_t$.
\end{defn}

We give some useful properties of this map. In particular, it is a monoid morphism $LU \to \Om U$ that preserves the CS-forms. 

\begin{lem} \label{lem;basedgt}
The based loop map induces a well defined map $LU \to \Om U$, denoted by $g_t \mapsto {}_*g_t$ as above, and satisfies
\[
CS( {}_*g_t ) = CS(g_t) \quad \textrm{and} \quad CS({}_*g_t  \oplus {}_* h_t) = CS({}_*(g_t \oplus h_t)).
\]
Furthermore, the map $g_t\mapsto {}_* g_t$ is a homomorphism with respect to the shuffle block sum, i.e. for $g_t, h_t:M\to LU^n_{-n}$, we have
\[
{}_*(g_t \boxplus h_t) = {}_*g_t \boxplus {}_* h_t.
\]
\end{lem}
\begin{proof}
 According to Lemma 3.6 of \cite{TWZ3} we have $CS( \gamma_g(t)) = 0$ so the first result follows since  $CS(g_t \oplus g_0^{-1}) = CS(g_t) + CS(g_0^{-1}) = CS(g_t)$.  The second  statement now follows formally since
 \begin{eqnarray*}
 CS({}_*g_t  \oplus {}_* h_t) &=& CS({}_*g_t) + CS( {}_* h_t)=CS(g_t) + CS(h_t) \\ 
 &=& CS(g_t \oplus h_t))=CS({}_*(g_t \oplus h_t)).
 \end{eqnarray*}
The last statement follows using two facts. First, the operations $\boxplus$ and $\oplus$ satisfy the interchange law for maps $g_t, h_t, k_t, l_t:M\to LU^n_{-n}$, \emph{i.e.}
\[
(g_t \oplus h_t) \boxplus (k_t \oplus l_t) = (g_t\boxplus k_t)  \oplus ( h_t \boxplus l_t)
\]
and, secondly, since the shuffle sum preserves unitary matrices, and the inverse of a unitary matrix is the conjugate transpose, we have $(g \boxplus h)^{-1} = (g \boxplus h)^* = g^* \boxplus h^* = g^{-1} \boxplus h^{-1}$. Therefore, using \eqref{EQU:gamma-square-sum},
\begin{multline*}
{}_*(g_t \boxplus h_t) 
=
 \gamma_{(g \boxplus h)_0}(t) * ((g_t \boxplus h_t)  \oplus ((g \boxplus h)_0)^{-1})  * \overline{\gamma_{(g \boxplus h)_0}(t)}
\\
=
\Big(\gamma_{g_0}(t) \boxplus \gamma_{h_0}(t) \Big) * \Big( (g_t \oplus g_0^{-1})  \boxplus (h_t \oplus h_0^{-1}) \Big) * \Big( \overline{\gamma_{g_0}(t)}  \boxplus \overline{\gamma_{h_0}(t)} \Big)
= {}_*g_t \boxplus {}_* h_t.
\end{multline*}
 \end{proof}

 We are now ready to define the $S^1$-integration map.  We fix once and for all a (continuous) homotopy inverse to the Bott Periodicity map $E$, \emph{i.e.} a map\label{map:i} $i: \Om U \to BU \times \Z$, and a homotopy $F_r: \Om U \times I \to \Om U$ satisfying $F_0 = E \circ i$  and $F_1 = id$, and a homotopy $H_r: BU \times \Z \times I \to BU \times \Z$ such that $H_0 = i \circ E$ and $H_1 = id$. 
 
 It is worth noting that if we could choose $i$, $F_r$ and $H_r$ to be smooth, then the presentation here could be simplified considerably.\footnote{There is a  geometric  map from $\Om U$ to the space of Fredholm operators on a separable Hilbert space, extensively studied in \cite{PS}. It would be interesting and useful if this map can be used in the way mentioned above.} Such a smooth choice is 
 unknown to us at the time of this writing, nevertheless we'll overcome this by replacing our maps by certain smooth maps from $M$, and then showing our constructions are independent of the choices.

\begin{defn} \label{defn;S1int} We now define the map $\II: \hat K^{-1}(M \times S^1) \to \hat K^0 (M)$. Let $g_t:M\times S^1\to U$ be a representative for an element $[g_t]\in K^{-1}(M \times S^1)$, and denote by ${}_*g_t:M\to \Omega U$ the based map from Definition \ref{def;basedgt}. Composing this with the fixed (continuous) homotopy inverse $i:\Omega U\to BU\times \Z$ to the Bott Periodicity map $E:BU\times \Z\to \Omega U$ gives a (continuous) map $i \circ {}_*g_t:M\to  BU\times \Z$.

To define $\II$, we make two choices. First, choose a smooth map $h: M \to BU \times \Z$ which is homotopic to $i \circ {}_*g_t$, via some continuous homotopy $h_r: M \times I \to BU \times \Z$, so
that $h_0 = h$ and $h_1 = i \circ {}_*g_t$. Note that $( E \circ h_r ) * (F_r \circ {}_* g_t)$ is then a continuous homotopy from
$E \circ h$ to ${}_* g_t$.

Second, choose a smooth homotopy $k_r:M\times I\to \Omega U$ between the smooth maps $k_0=E\circ h$ and $k_1={}_*g_t$
\[
\xymatrix{
M  \ar[rr]^{{}_* g_t} \ar@/^2pc/[rrrr]^{h}& & \Om U \ar[rr]^{i} & & \ar@/^2pc/[ll]^{E} BU\times \Z
}
\]
With these choices, denote by $\eta\in \Om^{\textrm{odd}}(M)$ the form
\[
 \eta =  \int_r k_r^* CS + h^* \beta 
 \]
and define $\II(g_t)\in \hat K^0(M)$ by
\[
\II(g_t) = [h] + a(\eta),
\]
where $[h]$ indicates the $CS$-equivalence class of the smooth map $h: M \to BU \times \Z$, and $a: \Om^{\textrm{odd}}(M) / Im(d) \to \hat K^{0}(M)$ is the map from Definition \ref{DEF:even-diff-K-theory-R-a-maps}.
\end{defn}

\begin{rmk} \label{rmk;k_rHasNiceh_r}
We note that the choice of homotopy $h_r$ is not part of the data in the above definition. In fact, for any chosen $h$ homotopic to $i\circ {}_* g_t$, and any smooth homotopy $k_r:M\times I\to \Om U$ between $k_0=E\circ h$ and $k_1={}_* g_t$, there always exists a homotopy $h_r$ from $h_0=h$ to $h_1=i\circ {}_* g_t$ such that $k_r$ is homotopic to $( E \circ h_r ) * (F_r \circ {}_* g_t)$ relative to the endpoints $E\circ h$ and ${}_*g_t$.
To find the homotopy $h_r$ from $h$ to $i \circ {}_* g_t$ with the desired property, we first choose any homotopy $h'_r$ from $h'_0=h$ to $h'_1=i\circ {}_* g_t$. Then, for any homotopy $k_r$ from $k_0 = E \circ h$ to $k_1 = {}_* g_t$, the path composition $k_r * (F_{1-r} \circ {}_* g_t)*(E\circ h'_{1-r})$ is a homotopy from $E\circ h$ to itself, and it thus defines an element in the homotopy group $\pi_1((\Om U)^M,E\circ h)$. Since the induced map $E^M:(BU\times \Z)^M\to (\Om U)^M$ is a homotopy equivalence, there is an element $h''_r\in \pi_1((BU\times \Z)^M, h)$ such that $E^M(h''_r)=E\circ h''_r$ is homotopic to $k_r * (F_{1-r} \circ {}_* g_t)*(E\circ h'_{1-r})$ relative to the endpoints.
\[
\begin{pspicture}(-1,0)(4,3)
%\psgrid(0,0)(3,3)
\pspolygon(.4,.4)(2.6,.4)(2.6,2.6)(.4,2.6)(.4,.4)
\rput[r](0.2,0.2){$E\circ h$}
\rput[r](0.2,1.5){$E\circ h''_r$}
\rput[r](0.2,2.8){$E\circ h$}
\rput(1.5,0.1){$k_r$}
\rput[l](2.8,0.2){${}_*g_t$}
\rput(1.5,2.9){$E\circ h'_r$}
\rput[l](2.8,2.8){$E\circ i \circ {}_*g_t$}
\rput[l](2.8,1.5){$F_r\circ {}_*g_t$}
  \end{pspicture} 
\]
But this means in turn that $k_r$ is homotopic to $(E\circ h''_r)*(E\circ h'_r)*(F_r\circ {}_* g_t)$ relative to the endpoints. Choosing $h_r=h''_r*h'_r$, we get the desired map from $h_0=h$ to $h_1=i\circ {}_* g_t$, such that $k_r$ is homotopic to $( E \circ h_r ) * (F_r \circ {}_* g_t)$ relative to the endpoints.

\end{rmk}

In the following Lemmas \ref{LEM:odd-well-def}-\ref{LEM:odd-I-commute-with-a} we will show that $\II$ is well-defined, and satisfies all the properties of an $S^1$-integration from Definition \ref{defn:S^1int}.

\begin{lem}\label{LEM:odd-well-def}
The map $\II: \hat K^{-1}(M \times S^1) \to \hat K^0 (M)$ is well-defined.
\end{lem} 
\begin{proof}

We first show that $\II(g_t) = [h] + a(\eta)$ is independent of the choices of $h$ and the homotopy $k_r$ that are used to define $\II(g_t)$. Given $g_t: M \times S^1 \to U$, suppose we have two choices of smooth maps, $h^0$ and $h^1$, which are both continuously homotopic to $ i {}_*g_t$, via some homotopies $h^0_r$ and $h^1_r$, respectively. Then $h^0$ and $h^1$ are smooth and continuously homotopic via the continuous homotopy $h^0_r * h^1_{1-r}$. So we can choose a smooth homotopy $s_r:M\times I \to BU\times \Z$ between $h^0$ and $h^1$, which is a deformation of the continuous homotopy $h^0_r * h^1_{1-r}$. This means the triangle in Figure \ref{fig;Triangles} can be filled in by a some smooth homotopy $T_1$, and so $E \circ T_1$ is a homotopy between $E \circ s_r$ and $(E \circ h^0_r) * (E \circ h^1_{1-r})$, as indicated in Figure \ref{fig;Triangles}.

\begin{figure}[h]
\[
\begin{pspicture}(0,-.5)(11,3.5)
%\psgrid(0,0)(11,3.5)
\pspolygon(0,0)(3,1.5)(0,3)(0,0)
\rput[l](-0.4,0){$h^1$}
\rput[l](-.4,3){$h^0$}
\rput[l](3,1.5){$i {}_* g_t$}
 \rput[l](-.4,1.5){$s_r$}
 \rput[r](1.7,.4){$h^1_r$}
 \rput[r](1.7,2.6){$h^0_r$}
 \pspolygon(5,0)(11,1.5)(5,3)(5,0)
 \pspolygon(5,0)(8,1.5)(5,3)
  \pspolygon(8,1.5)(11,1.5)(8,1.5)
\rput[l](4,0){$E \circ h^1$}
\rput[l](4,3){$E \circ h^0$}
\rput[r](7.6,1.5){$E \circ i {}_* g_t$}
 \rput[l](4,1.5){$E \circ s_r$}
 \rput[l](5.35, 1){$E \circ h^1_r$}
 \rput[l](5.35, 2){$E \circ h^0_r$}
 \rput[l](8,2.5){$k_r^0$}
 \rput[l](8,0.5){$k_r^1$}
 \rput[l](8.3,1.7){$F_r \circ {}_* g_t$}
 \rput[l](11.1,1.5){${}_* g_t$}
  \end{pspicture} 
\]
\caption{Homotopies between  two possible choices. Triangle $T_1$ with vertices $h^0, h^1$, and $i {}_* g_t$, and 
triangle $T_2$ with vertices $E \circ h^0, E \circ h^1$, and ${}_* g_t$.}
\label{fig;Triangles}
\end{figure}
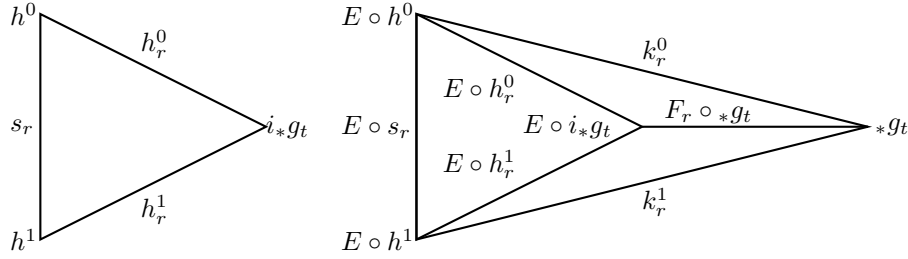

The homotopy $(E \circ h^0_r) * ( F_r \circ g_t)$ is a continuous homotopy from $E \circ h^0$ to ${}_* g_t$. We can deform this to a smooth homotopy $k^0_r$, and  similarly we can deform $(E \circ h^1_r) * ( F_r \circ g_t)$ to a smooth homotopy $k^1_r$.
This shows the second triangle $T_2$ in Figure \ref{fig;Triangles} can be filled in, and we may as well assume it is filled in by a smooth map, which we also denote by $T_2$.

 For $i=0,1$, let 
 \[
 \eta^i =  \int_r (k^i_r)^* CS + (h^i)^* \beta.
 \]
By definition of the ``$a$'' map (see Definition \ref{DEF:even-diff-K-theory-R-a-maps}), there is a path $\gamma_r : M \times I \to  BU \times \Z$ such that $a(\eta^0 ) = [\gamma_1]$,  $CS(\gamma_r) = \eta^0$, and $\gamma_0$ is the constant map to the identity. Similarly, there is a path $\rho_r : M \times I \to BU \times \Z$ such that $a(\eta^1 ) = [\rho_1]$, $CS(\rho_r) = \eta^1$, and $\rho_0$ is the constant map to the identity.

Then $s_r \boxplus (\rho_r \boxplus \gamma_{1-r})$ is a path from a representative of $[h^0 \boxplus (\rho_0 \boxplus \gamma_{1})]=[h^0] + a(\eta^0)$ to a representative of $[h^1 \boxplus (\rho_1 \boxplus \gamma_{0})]=[h^1] + a(\eta^1)$, and it suffices to show that the CS-form of this path is exact. Note that the CS form is given  by 
\begin{multline*}
CS(s_r \boxplus (\rho_r \boxplus \gamma_{1-r})) = CS(s_r) + CS(\rho_r) - CS(\gamma_r) 
=
\int_r s_r^* Ch + \eta^1 -\eta^0
\\
= \int_r s_r^* Ch  + \int_r (k^1_r)^* CS + (h^1)^* \beta -  \int_r (k^0_r)^* CS - (h^0)^* \beta.
\end{multline*}
Using the fact that $d \beta = E^* CS - Ch$ we have
\[
\int_r s_r^* Ch =   \int_r s_r^* E^* CS - \int_r s_r^* d \beta =  \int_r (E \circ s_r)^* CS + d\int_r s_r^*  \beta -  (h^1)^* \beta + (h^0)^* \beta.
\]
So that, modulo exact forms, 
\[
CS(s_r \boxplus (\rho_r \boxplus \gamma_{1-r})) =    \int_r (E \circ s_r)^* CS + \int_r (k^1_r)^* CS -  \int_r (k^0_r)^* CS =  d  \iint T_2^* CS
\]
where $T_2$ is the smooth homotopy filling the second triangle, as in Figure \ref{fig;Triangles}. This shows that $\II$ is independent of the choices made.

It remains to show that if $g_{t,0}$ and $g_{t,1}$ are $CS$-equivalent then $\II(g_{t,0}) = \II (g_{t,1})$. Let $g_{t,r} : M \times S^1 \times I \to U$ for $t \in S^1$, $r \in I$ be a smooth homotopy from $g_{t,0}$ to $g_{t,1}$ such that $CS(g_{t,r})$ is exact. By definition of the map $\II$ we must first choose smooth maps $h^0$ and $h^1$ which are homotopic to $i \circ {}_* g_{t,0}$ and $i \circ {}_* g_{t,1}$, respectively. Since $g_{t,0}$ and $g_{t,1}$ are homotopic via the smooth homotopy $g_{t,r}$ above, it follows that $i \circ {}_* g_{t,0}$ and $i \circ {}_* g_{t,1}$ are (continuously) homotopic via the homotopy $i \circ {}_* g_{t,r}$. Since we have already shown that $\II(g_t)$ is independent of the chosen representative $h$, we may as well choose the same map $h$ for both $g_{t,0}$ and $g_{t,1}$. Next, we must choose a smooth homotopy $k_{r,0}:M\times I\to \Om U$ from $E \circ h$ to $ {}_*  g_{t,0}$, and a smooth homotopy $k_{r,1}:M\times I\to \Om U$ from $E \circ h$ to $ {}_*  g_{t,1}$. If we choose any homotopy $k_{r,0}$ from $E \circ h$ to $ {}_*  g_{t,0}$, then we may pick $k_{r,1}$ to be the path composition $k_{r,1}=k_{r,0} * ({}_*g_{t,r})$, since we have already shown that $\II(g_{t,1})$ is independent of the choice for $k_{r,1}$.

Now, to show that $\II(g_{t,0}) = [h] + a(\eta^0)$ and $\II(g_{t,1}) =  [h] + a(\eta^1)$ are  equal, it suffices to show $\eta^0 = \eta^1$ mod exact, since the map $a$ vanishes on exact forms (\emph{cf.} Definition \ref{DEF:even-diff-K-theory-R-a-maps}). But for the chosen $k_{r,1}$, we get
\[
\eta^1 = \int_r (k_{r,1})^* CS + h^* \beta = \int_r (k_{r,0})^*CS + \int_r ({}_* g_{t,r})^* CS + h^* \beta  = \eta^0 + \int_r  ({}_* g_{t,r})^* CS.
\]
By Lemma \ref{lem;basedgt} we have 
\[
\int_r ({}_* g_{t,r})^* CS = \int_r g_{t,r}^* CS = \int_{ r\in I}   \int_{t \in S^1}  g_{t,r}^* Ch =  \int_{t \in S^1}  \int_{ r\in I} g_{t,r}^* Ch  = \int_{t \in S^1} CS(g_{t,r}).
\]
But $CS(g_{t,r})$ is exact by assumption and $d$ commutes with the integral over $S^1$ since $S^1$ is closed, showing that $\eta^1$ and $\eta^0$ only differ by an exact form.
 
This completes the proof that $\II$ is well defined.
\end{proof}

\begin{lem}\label{LEM:I-group-homo}
The map  $\II : \hat K^{-1}(M \times S^1) \to \hat K^0(M)$ is a group homomorphism.
\end{lem}
\begin{proof}
Recall that the group structure ``$+$'' on $\hat K^*(M)$ is induced on representatives by ``$\boxplus$'' from Definitions \ref{Usum2} and \ref{defn;BUsum}. Let $g_t$ and $f_t$ be representatives for elements in $ \hat K^{-1}(M \times S^1)$, and denote with the induced based maps ${}_* g_t, {}_* f_t : M \to \Om U$. We must show that 
\[
\II(g_t \boxplus f_t) = [u] + a( \eta^{l_r})
\]
is $CS$-equivalent to 
\[
\II(  g_t ) + \II(f_t) = [s] + [t] + a( \eta^{k_r} )+ a( \eta^{h_r} ).
\]
where $u$ is smooth and homotopic to $i \circ ({}_* g_t  \boxplus {}_* f_t)$ via some homotopy $u_r$, 
$s$ is smooth and homotopic to $i \circ {}_* g_t$  via some homotopy $s_r$, 
and $t$ is smooth and homotopic to $i \circ {}_* f_t$ via some homotopy $t_r$. Recall that
\begin{eqnarray*} 
\eta^{l_r} &=&  \int_r (l_r)^* CS + u^* \beta \\
\eta^{k_r} &=&  \int_r (k_r)^* CS + s^* \beta  \\
\eta^{h_r} &=&  \int_r (h_r)^* CS + t^* \beta 
\end{eqnarray*}
where $l_r, k_r, h_r$ are smooth and, by Remark \ref{rmk;k_rHasNiceh_r}, we may assume are homotopic to $(E \circ u_r) * (F_r \circ ({}_* g_t \boxplus {}_*f_t))$, 
$(E \circ s_r) * (F_r \circ {}_* g_t )$, and
$(E \circ t_r) * (F_r \circ {}_* f_t)$, respectively.

We choose representatives for $ a(\eta^{l_r})$, $a(\eta^{k_r})$ and $a(\eta^{h_r})
\in \hat K^0(M)$ and denote these by $\gamma_1$ , $\kappa _1$, and $\rho _1:M\to \B$, respectively. By definition of the ``$a$'' map (as the inverse of the CS form obtained by a path to the identity),  there are maps 
$\gamma_r , \kappa _r, \rho _r : M \times I \to BU \times \Z$ 
satisfying $\gamma_0 = \kappa_0 = \rho_0 = id$ and 
$CS(\gamma_r) = \eta^{l_r}$, $CS(\kappa_r) =  \eta^{k_r}$, and 
$CS(\rho_r) =  \eta^{h_r}$.

The proof will be an adaptation of an argument by Upmeier in \cite{U}.

Consider the following diagram
\[
\xymatrix{
M \ar[rd]_-{{}_* g_t \times  {}_*f_t}  \ar `r[rrr]^{u} `[dd][rrdd] \ar `d[rdd][rdd]_-{s \times t\hspace{.7in}} 
&&& \\
 & \Om U \times \Om U \ar[d]_-{i \times i} \ar[r]^-{\boxplus} & \Om U \ar[d]^i & \\  & (BU \times \Z) \times (BU \times \Z) \ar[r]^-\boxplus & BU \times \Z  &
}
\]
There is a homotopy $K:  \Om U \times \Om U  \times I \to BU \times \Z$ given by 
\[
K(r) = 
 \left\{ 
\begin{array}{ll} 
i( F_r \boxplus F_r ),   & \text{if } r \in [0,1/2]\\ 
 H_r ( i  \boxplus  i  ),  & \text{if } r \in[1/2,1] 
\end{array} \right.
\]
where $F_r: \Om U \times [0,1/2] \to \Om U$ is as before (suitably reparametrized for better readability), satisfying $F_0 = id$ and $F_{1/2}= E \circ i$, and $H : BU \times \Z \times [1/2,1] \to BU \times \Z$ satisfies $H_{1/2} = i \circ E$ and $H_1 = id$ so that $K(0) = i \circ  \boxplus  $ and $K(1) = i \boxplus i$. Note this is well defined at $r=1/2$ since $(A-cI_0)\boxplus(B-cI_0)=(A\boxplus B)-cI_0$ and $\exp(A\boxplus B)=\exp(A)\boxplus \exp(B)$, so that $E(P\boxplus Q)=E(P)\boxplus E(Q)$.

This implies there is a smooth homotopy between $u$ and $s \boxplus t$, since  $u$ and $s \boxplus t$ are homotopic via the continuous
homotopy 
\[
C(r)= u_{r} * (K(r) \circ ({}_* g_t \times {}_* f_t) ) * (s_{1-r} \boxplus t_{1-r}).
\]

For a choice of a smooth homotopy $\alpha_r$ from $u$ to $s \boxplus t$, we must show that the smooth homotopy $\alpha_r \boxplus (\gamma_{1-r}  \boxplus (\kappa_{r}  \boxplus \rho_{r}))$, from the  representative $u \boxplus  (\gamma_1\boxplus id)$ of $[u]+a(\eta^{l_r})$ to the representative $(s \boxplus t) \boxplus (id\boxplus(\kappa_1 \boxplus \rho_1))$ of $[s]+[t]+a(\eta^{k_r})+a(\eta^{h_r})$,  has an exact CS-form.
This CS-form is in fact equal to 
\[
CS(\alpha_r) - \eta^{l_r} + \eta^{k_r} + \eta^{h_r} = CS(\alpha_r)  -\int_r l_r^* CS - u^* \beta + \int_r k_r^* CS + s^* \beta  + \int_r h_r^* CS + t^* \beta 
\]
In order to show that this form is exact, it is enough to show that the form is exact thought of as a cocycle, via the deRham Theorem.

Choose a nice cochain model with a representative for integration along the interval $I$. For example we may take the cubical singular model, and then for a cochain $c$ we can represent integration along the interval by the slant product $c \mapsto c \setminus [I]$, where $[I]$ is the fundamental chain of the interval. Notice that the deRham map $\omega \mapsto  dR(\omega) = \int \omega$ commutes with the slant product and integration along the interval $I$, \emph{i.e.} $dR( \int_I  \omega) = dR(\omega) \setminus [I]$.

Now, for any choice of smooth $\alpha_r$ continuously homotopic to $C(r)$ (relative to the boundary smooth maps $u$ and $s \boxplus t$) we have that the cochains $dR(CS(\alpha_r))$ and $C(r)^*(dR(CS))$ differ by a coboundary. In fact, if $\alpha_{r,s}$ is a continuous relative homotopy from $\alpha_r$ to $C(r)$ fixing the maps at the boundary,  then integrating over $s \in I$ we have 
\[
\delta ( \alpha_{r,s}^* dR(CS) \setminus [I]) =  \alpha_{r,1}^* dR(CS) - \alpha_{r,0}^* dR(CS)=\alpha_{r}^* dR(CS) - C(r)^* dR(CS)
\]
since $\delta dR(CS) = 0$.
So, working modulo exact cochains, we can replace $CS(\alpha_r)$ by
\[
C(r)^*(dR(CS)) = dR(CS(u_r)) + (K(r) \circ ({}_* g_t \times {}_* f_t) )^*(dR(CS)) - dR(CS(s_{r} \boxplus t_{r}))
\]
and it suffices to show that 
\begin{multline*}
X =
u_r^*(dR(CS)) + (K(r) \circ ({}_* g_t \times {}_* f_t) )^*(dR(CS)) - s_{r} ^*(dR(CS))-t_{r}^*(dR(CS))
\\
+dR\left(-\int_r l_r^* CS - u^* \beta
+  \int_r k_r^* CS + s^* \beta  + \int_r h_r^* CS + t^* \beta \right)
\end{multline*}
is an exact cochain, where we used that $CS(s_r\boxplus t_r)=CS(s_r)+CS(t_r)$.
Now, adding to $X$ the exact cochains
\begin{eqnarray*}
-\delta \Big((u_r^*dR( \beta))\setminus[I]\Big) & = &  - (i \circ ({}_* g_t \boxplus {}_* f_t))^* dR(\beta) + u^*dR(\beta)  + u_r^*dR(E^* CS - Ch)\setminus [I]\\
\delta \Big((s_r^*dR( \beta))\setminus[I]\Big) & = &   (i \circ {}_* g_t )^* dR(\beta) - s^*dR(\beta)  -  s_r^*dR(E^* CS - Ch)\setminus[I]\\
\delta \Big((t_r^*dR( \beta))\setminus[I]\Big) & = &   (i \circ {}_* f_t )^* dR(\beta) - t^*dR(\beta)  -  t_r^*dR(E^* CS - Ch)\setminus[I]
\end{eqnarray*}
and using the relations 
\begin{eqnarray*}
-dR(\int_r l_r^* CS) &=& - \Big((E \circ u_r)^* dR(CS) \Big)\setminus[I]- \Big( (F_r \circ ( {}_* g_t \boxplus {}_*f_t) )^* dR(CS)\Big) \setminus[I] \\
 dR(\int_r k_r^* CS) &=& \Big( (E \circ s_r)^* dR(CS) \Big)\setminus[I] + \Big( (F_r \circ  {}_* g_t)  ^* dR(CS) \setminus[I] \Big)\\
dR (\int_r h_r^* CS) &=&\Big( (E \circ t_r)^* dR(CS) \Big) \setminus[I] + \Big((F_r \circ  {}_*f_t) ^* dR(CS)\Big) \setminus[I]
\end{eqnarray*}
which follow from our assumptions on  $l_r, k_r$ and  $h_r$, we obtain, modulo exact, that 
\begin{eqnarray*}
X 
&=& (K(r) \circ ({}_* g_t \times {}_* f_t) )^*(dR(CS))\\
&& - (i \circ ({}_* g_t \boxplus {}_* f_t))^* dR(\beta)
  + (i \circ {}_* g_t)^* dR( \beta) 
  + (i \circ {}_* f_t)^* dR(\beta) \\
  &&
  -\Big( (F_r \circ ( {}_* g_t \boxplus {}_*f_t) )^* dR(CS)\Big) \setminus[I]\\
  &&
  +\Big( (F_r \circ  {}_* g_t)  ^* dR(CS) \setminus[I] \Big)
  +\Big((F_r \circ  {}_*f_t) ^* dR(CS)\Big) \setminus[I].
\end{eqnarray*}
Notice that this can we written as $X= ({}_* g_t \times {}_* f_t)^* (Y)$ for $Y\in C^{\odd}(\Om U\times \Om U)$ defined to be
\begin{multline*}
Y = K(r)^*(dR(CS))  - (i \circ \boxplus ) ^* dR(\beta)    +
 (i \circ pr_1) ^* dR(\beta)  +  (i \circ pr_2)^* dR(\beta) \\
 - \Big((F_r \circ \boxplus)^* dR(CS) \Big)  \setminus[I]
 +\Big( (F_r \circ pr_1) ^* dR(CS) \Big)  \setminus[I]
 + \Big( (F_r \circ  pr_2) ^* dR(CS)\Big)  \setminus[I],
\end{multline*}
where $pr_1, pr_2: \Om U \times \Om U \to \Om U$ are the projections onto the first and second factors, respectively. To check that $X$ is exact it suffices to show that $Y$ is exact, since the pullback $({}_* g_t \times {}_* f_t)^*$ preserves exactness. Now, since $H^{\odd}(\Om U \times \Om U) = 0$, it suffices to show that $Y$ is closed. We check this using the facts that $\delta dR( \beta) = E^* dR(CS) - dR(Ch)$, $F_r$ is a homotopy from $F_0= E \circ i$ to $F_1 = id$, $CS$ is closed, and $K(r)$ is as above, so that 
\begin{eqnarray*}
\delta \Big(K(r)^*dR(CS)\Big) &=& \delta \Big(\big(K(r)^*dR(Ch)\big)\setminus[I]\Big)
=K(1)^* dR(Ch)-K(0)^* dR(Ch) \\
&=& (i \times i)^* \circ \boxplus^* (dR(Ch))-  \boxplus^*  \circ i^* (dR(Ch)).
\end{eqnarray*}
With this we obtain
\begin{eqnarray*}
\delta Y&=& (i \times i)^* \circ \boxplus^* (dR(Ch))- \boxplus^*  \circ i^* (dR(Ch)) -(i\circ \boxplus)^* (E^* dR(CS)-dR(Ch))\\
&& +(i\circ pr_1)^*(E^* dR(CS)-dR(Ch))+(i\circ pr_2)^*(E^* dR(CS)-dR(Ch)) \\
&& -(\boxplus-E\circ i\circ \boxplus)^*dR(CS)
+(pr_1-E\circ i\circ pr_1)^*dR(CS)\\
&&+(pr_2-E\circ i\circ pr_2)^*dR(CS)\\
&=&
(\boxplus\circ (i \times i)-i\circ pr_1-i\circ pr_2)^*dR(Ch)  -(\boxplus-pr_1-pr_2)^*dR(CS)\\
&=&
(i\times i)^*\circ (\boxplus- pr_1- pr_2)^*dR(Ch)  -(\boxplus-pr_1-pr_2)^*dR(CS)\\
&=&0,
\end{eqnarray*}
where in the last equality we used that, in general, $Ch(x\boxplus y)=Ch(x)+Ch(y)$ and $CS(x_t\boxplus y_t)=CS(x_t)+CS(y_t)$.

This completes the proof that $\II$ is a group homomorphism.
\end{proof}

\begin{lem}
The map $\II$ satisfies $\II \circ ( id \times r)^* = - \II$ where $r :S^1 \to S^1$ is given by $r(z) = \bar z = z^{-1}$. 
\end{lem}
\begin{proof}
Note that for a map $g_t : M \times S^1 \to U$ we have that
$ ( id \times r)^* g_t = \bar g_t$, where $\bar g_t$ is the reversed loop. 
The desired condition holds if and only if $\II(g_t \boxplus \bar g_t ) = 0$ since 
\[
\II(g_t) + \II(\bar g_t) = \II( g_t \boxplus \bar g_t ).
\]

Let $\phi: \Om U \to \Om U$ be given by $\phi(k_t) =  k_t \boxplus \bar k_t$ and let
$f_r: \Om U \times I \to \Om U$ be any (continuous) homotopy satisfying
\[
f_0 (k_t) =  1, \quad \text{ and } \quad \quad f_1(k_t)  = \phi(k_t).
\]
For example we may choose the composition of homotopies whose value on some $k_t$ is given by 
\[
k_t \boxplus \bar k_t = (k_t \boxplus 1) \cdot (1 \boxplus \bar k_t) \sim (k_t \boxplus 1) \cdot (\bar k_t \boxplus 1)  = (k_t \cdot \bar k_t \boxplus 1)  
\sim (k_t \circ \bar k_t \boxplus 1) \sim (1 \boxplus 1).
\]
For later use, we note that $\phi^*(CS)=0$, since for any plot $k_t:M\to \Om U$, we have
\begin{eqnarray*}
k_t^*(\phi^*(CS))
&=&
CS(\phi(k_t))
=CS(k_t\boxplus \bar{k_t})
\\
&=&
CS(k_t)+CS(\bar{k_t})
=\int_{S^1}Ch(k_t)+\int_{S^1}Ch(\bar{k_t})
=0.
\end{eqnarray*}

Since $\phi$ is homotopic via $f_r$ to the constant map, $i ({}_* g_t \boxplus \bar{ {}_*g_t} ) = i (\phi( {}_* g_t)) : M \to BU \times \Z$ is also homotopic to constant map. So, we may choose the map $h$ as in the definition of $\II$ to be the constant map, which is smooth. Let $k_r: M \times I \to \Om U$ be given by $k_r = f_r \circ ({}_* g_t\times id):M\times I\stackrel{{}_* g_t\times id}{\longrightarrow} \Om U\times I\stackrel{f_r}{\longrightarrow} \Om U$, so that $k_r$ is a homotopy from $k_0 = 1$ to $k_1=\phi( {}_* g_t)$.
This map may not be smooth, but we can choose relative homotopy to a map $k^s_r: M \times I \times I \to \Om U$
which extends the map $k^0_r = k_r$, fixing the smooth maps at the endpoints $k^0_0 = k^s_0 = 1$, and $k^0_1= k^s_1 =\phi({}_* g_t)$,  so that the map $k^1_r$ is smooth. 
\[
\begin{pspicture}(-1,0.3)(6,4.8)
%\psgrid(-1,0)(6,5)
\pspolygon(0,1)(5,1)(5,4)(0,4)
 \rput[r](-.1,1){$k_0^1=1$}
 \rput[r](-.1,2.5){$k_0^s=1$}
 \rput[r](-.2,4){$k_0^0=1$}
 \rput[l](5.1,1){$k_1^1=\phi({}_*g_t)$}
 \rput[l](5.1,4){$k_1^0=\phi({}_*g_t)$}
 \rput[l](5.1,2.5){$k_1^s=\phi({}_*g_t)$}
 \rput(2.5,.7){$k^1_r$ (smooth)}
 \rput(2.5,2.5){$k^s_r$}
 \rput(2.5,4.3){$k_r^0=k_r=f_r\circ({}_*g_t\times id)$}
  \end{pspicture} 
\]
Then we let
\begin{equation}\label{eta-in-reverse-proof}
\eta = \int_r (k^1_r)^* CS + h^* \beta =  \int_r (k^1_r)^* CS\quad\quad\in \Omega^{\textrm {odd}}(M)
\end{equation}
and by definition we have $\II(g_t \boxplus \bar g_t ) = [h] + a(\eta) = a(\eta)$, where the last equality follows since $h$ is constant. In order to show that $\II(g_t \boxplus \bar g_t ) =0$, it now suffices to show that $\eta$ is an exact form, since $a$ vanished on exact forms. (Note that $\eta$ is closed, since  $d \eta = (k^1_1)^* CS -  (k^1_0)^* CS = CS(\phi({}_* g_t)) - CS(const) = CS(\phi({}_* g_t))=0$.) We will show that $\eta$ is indeed exact by showing that the deRham form $\eta$, thought of as a cocycle via the deRham Theorem, is an exact cocycle.

As in the previous lemma, we choose a nice cochain model with a representative for integration along the interval $I$, so that for a cochain $c$ we can represent integration along the interval by the slant product $c \mapsto c \setminus [I]$, where $[I]$ is the fundamental chain of the interval. Recall that the deRham map $\omega \mapsto  dR(\omega) = \int \omega$ commutes with the slant product and integration along the interval $I$, \emph{i.e.} $dR( \int_I  \omega) = dR(\omega) \setminus [I]$.

Consider the cochain 
\[
X = \left( (k^s_r)^* dR(CS)  \right)  \setminus [I \times I]
\]
where $ \setminus [I \times I] =  \setminus [I]  \, ( \, \setminus [I] )$.
Then
\[
\delta X =  \left( (k^1_r)^* dR(CS)  \right)  \setminus [I] -  \left( (k^0_r)^* dR(CS)  \right)  \setminus [I] -  \left( (k^s_1)^* dR(CS)  \right)  \setminus [I] -  \left( (k^s_0)^* dR(CS)  \right)  \setminus [I]
\]
Since $k^s_0 = 1$, the last term vanishes. Also, the second to last term vanishes since $k^s_1 =\phi({}_* g_t)$ and $CS(\phi({}_* g_t)) =0$, as shown above.
This shows that
\begin{eqnarray*}
\delta X &=&  \left( (k^1_r)^* dR(CS)  \right)  \setminus [I] -  \left( (k^0_r)^* dR(CS)  \right)  \setminus [I]
\\
&=& dR \Big(\int_r (k^1_r)^* CS\Big) -   \left( (f_r (\circ {}_* g_t\times id))^* dR(CS)  \right)  \setminus [I],
\end{eqnarray*}
since $k^1_r$ is smooth, and for cochains in the image of the deRham map, the slant product equals the integration along $I$. By equation \eqref{eta-in-reverse-proof}, it thus suffices to show that the last term is exact. But
\begin{eqnarray*}
 \left( (f_r \circ ({}_* g_t\times id))^* dR(CS)  \right)  \setminus [I] &=&  \left( ({}_* g_t\times id)^* \circ f_r^*dR(CS)  \right)  \setminus [I]
 \\
& = & ({}_* g_t)^*  \left( f_r^*dR(CS)  \setminus [I]  \right) 
\end{eqnarray*}
and $ f_r^*dR(CS)  \setminus [I]$ is a closed odd degree cochain on $\Om U$, since
\[
\delta \left( f_r^*dR(CS)  \setminus [I]  \right)  = f_1^* dR(CS) - f_0^* dR(CS) = dR(\phi^* CS) - dR(const^* CS) = 0.
\]
Since $H^\textrm{odd}(\Om U) = 0$, the term $ f_r^*dR(CS)  \setminus [I]$ must also be exact, showing that $({}_* g_t)^*  \left( f_r^*dR(CS)  \setminus [I]  \right)$ is exact as well, and so we are done.
\end{proof}

\begin{lem}
The integration map $\II: \hat K^{-1}(M \times S^1) \to \hat K^0(M)$ satisfies $\II \circ p^* = 0$, where $p : M \times S^1 \to M$ is projection.
\end{lem}
\begin{proof}
Let $g: M \to U$ be a representative for $[g] \in \hat K^{-1}(M)$. Then $g_t : M\times S^1 \to U$ 
defined by $g_t = p^*g$ satisfies $g_t = g$ for all $t \in S^1$. Recall from page \pageref{path:gamma_g}, that $\gamma_g:M\times I\to U$ is a path with $\gamma_g(0)=g\oplus g^{-1}$ and $\gamma_g(1)=id$. It follows that the induced based map ${}_* g_t : M \to \Om U$ 
defined by 
\[
{}_*g_t = \gamma_g (t) * (g_t \oplus g_0^{-1}) * \overline{\gamma_g (t) }
\] 
is the path from $id$ to   $g \oplus g^{-1}$ and back to $id$ (\emph{cf.} Lemma \ref{lem;basedgt}).
We must show $\II(g_t) =0$.

Denote by $\gamma_g^s(t)=\gamma_g(1-s(1-t))$ the path from $\gamma_g(1-s)$ to $\gamma_g(1)=id$, and consider the homotopy $h_{t,s} : M \times I \to \Om U $ which contracts ${}_* g_t$ to the identity, which is defined by concatenating the following paths,
\[
h_{t,s} =  \gamma_g^s * const_{\gamma_g(1-s)} * \overline{\gamma_g^s}, \quad \text{for }s\in I.
\]
Then $h_{t,1} =\gamma_g* (g\oplus g^{-1}) *\overline{\gamma_g}=  {}_* g_t$, and $h_{t,0}  : M \to  \Om U$ is the constant map to the identity in $U$.

Since ${}_* g_t$ is homtopic to a constant, we can choose (in the definition of $\II$) $h: M \to BU \times \Z$ to be the constant map.
Then $E \circ h$ is also constant and we can let $k_s = h_{t,s}$ be the chosen homotopy from $k_1 = {}_* g_t$ to $k_0 = E \circ h$, the constant map.
Then
\begin{eqnarray*}
\eta &=& \int_s  (k_s)^* CS + h^* \beta = \int_s  (h_{t,s})^* CS \\ &=& CS(\gamma_g^s ) +CS(const)+ CS(\overline{\gamma_g^s} ) = CS(\gamma_g^s )- CS(\gamma_g^s )=0.
\end{eqnarray*}
So that
\[
\II (g_t) = [h] + a(\eta) = 0
\]
\end{proof}

\begin{lem} \label{lem;Chcommutes}
The map $\II$ makes the following diagram commute
\[
\xymatrix{
\hat K^{-1}(M \times S^1) \ar[r]^{Ch} \ar[d]^{\II} &  \Om^{\textrm{odd}}(M \times S^1) \ar[d]^{\int_{S^1}} \\
 \hat K^0 (M)  \ar[r]^-{Ch} &    \Om^{\textrm{even}}(M)
}
\]
\end{lem}
\begin{proof}
This follows by a direct computation. As in the definition of $\II$, choose $h$ homotopic to $i \circ {}_*g_t$ and choose 
 a homotopy $k_r$  such that $k_0 = E \circ h$ and $k_1=  {}_*  g_t$. Define $\eta  = \int_r k_r^* CS + h^* \beta$,  then by Lemma \ref{EQU:dbeta} we have
\begin{eqnarray*}
d \eta &=& k_1^* CS  - k_0^* CS + h^* (E^* CS - Ch) \\ 
&=&  {}_*  g_t^* CS - (E \circ h)^* CS  + h^* (E^* CS - Ch) =  g_t^* CS  - h^*  Ch 
 \end{eqnarray*}
so that
\begin{eqnarray*}
Ch ( \II( g_t) ) 
 &=& Ch  ([h]) + Ch( a( \eta)) = Ch  ([h]) +   d \eta \\ 
 &=&  Ch  ([h]) +  g_t^* CS  - h^*  Ch 
 = g_t^* CS = \int_t g_t^* Ch.
 \end{eqnarray*}
\end{proof}

\begin{lem}
The following diagram commutes.
\[
\xymatrix{
\hat K^{-1}(M \times S^1) \ar[r]^{I} \ar[d]^{\II}  & K^{-1}(M \times S^1) \ar[d]^{\int_{S^1}} \\
 \hat K^{0}(M ) \ar[r]^{I} & K^{0}(M ) 
}
\]
where the right vertical map is the $S^1$-integration map in $K$-theory.
\end{lem}
\begin{proof} As in the definition of $\II$, choose $h$ homotopic to $i \circ {}_*g_t$ and choose 
 a homotopy $k_r$  such that $k_0 = E \circ h$ and $k_1=  {}_*  g_t$. Define $\eta  = \int_r k_r^* CS + h^* \beta$.
Since $\II(g_t) = [h] + a(\eta)$, where $h$ is homotopic to $i \circ {}_*g_t$ and $Im(a) = Ker(I)$, it suffices to show for any representative $g_t$ that
\[
I ( i( {}_*g_t)  )= \int_{S^1} I ( g_t  ).
\]
Consider the following diagram, where the bottom row defines the $S^1$-integration in $K$-theory as defined in Definition \ref{def;S1intKtheory}.
\[
\xymatrix{
 [M, LU] \ar[r]^-{\pi} \ar[d]^{=} & [M,\Om U] \ar[rr]^{i_*} \ar[d]_{=} \ar[dr] & & [M, BU \times \Z] \ar[d]^{=} \\
K^{-1}(M \times S^1) \ar[r]^-{pr} & Ker(j^*) \ar[r]^-{(q^*)^{-1}} & \tilde K^{-1}(\Sigma M_+) =K^{-1}(M\wedge S^1) \ar[r]^-{\sigma^{-1}} & K^0(M)
}
\]
The map $\pi$ is defined as follows.
The adjoint of a map $M \times S^1\to U$ is a map $M \to LU$. Using the transport to the identity map, $g_t \mapsto g_0^{-1}g_t$, we have
$LU \cong U \times \Om U $, and we let $\pi$ denote projection onto the $\Om U$ factor. Note that since
$g_0^{-1}g_t$ is homotopic to ${}_*g_t$, we have that the induced map on homotopy classes
$\pi: [M, LU]  \to [M, \Om U]$ is the same as the map $g_t \mapsto {}_*g_t$. This shows that on homotopy classes we have
\[
(i_* \circ \pi) (g_t) =  i ( {}_* g_t ),
\]
so the top row of the diagram represents $I ( i ( {}_* g_t ))$.

The left square and middle triangle commute. So, it suffices to show the rightmost quadrilateral commutes. But $i_*$ is an isomorphism with inverse induced by the Bott periodicity map $E: BU \times \Z \to \Om U$, which induces the suspension isomorphism $\sigma$.
\end{proof}

\begin{lem}\label{LEM:odd-I-commute-with-a}
The following diagram commutes for all manifolds $M$.
\[
\xymatrix{
\Om^{\even}(M \times S^1)/Im(d) \ar[r]^-{a} \ar[d]^{\int_{S^1}} & \hat K^{-1}(M \times S^1)\ar[d]^{\II}\\
\Om^{\odd}(M) \ar[r]^-{a} /Im(d)& \hat K^0(M ) 
}
\]
\end{lem}
\begin{proof}
By definition of the maps $a:\Om^{*+1}(M)/Im(d) \to \hat K^*(M)$ we can write the diagram of interest
as 
\[
\xymatrix{
\Om^{\even}(M \times S^1)/Im(d) \ar[r]^-{pr} \ar[d]^{\int_{S^1}} &  \Om^{\even}(M \times S^1)/Im(Ch) 
\ar[d]^{\int_{S^1}} \ar[r]^-{CS^{-1}} \ar[d]^{\int_{S^1}}   &  Ker(I) \subset \hat K^{-1}(M \times S^1)\ar[d]^{\II}\\
\Om^{\odd}(M) \ar[r]^-{pr} /Im(d) & \Om^{\odd}(M) \ar[r]^-{CS^{-1}} /Im(Ch)  &  Ker(I) \subset  \hat K^0(M ) 
}
\]
where $pr$ is the projection, which is well defined since $Im(d) \subset Im(Ch)$ by Theorem \ref{THM:3-statements}\eqref{statement-1}, and the middle vertical map is well defined since $Im(\int_{S^1} \circ Ch) \subset Im(Ch)$, by Lemma \ref{lem;deta2} (which will be proven independently from this lemma) and Theorem \ref{THM:3-statements}\eqref{statement-1}. The left square clearly commutes, so it suffices to show that the right square commutes. Note, that the maps $CS^{-1}$ are isomorphisms onto $Ker(I)$.

For given $g_t \in Ker(I)\subset\hat K^{-1}(M\times S^1)$ there is a path $g_{t,s}: (M \times S^1) \times I \to U$ such that $g_{t,0}= id$ is constant for all $t\in S^1$, and $g_{t,1} = g_t: M \times S^1 \to U$. Note that the parameter for this path $g_{t,s}$ is $s\in I$, so that $CS(g_{t,s})=\int_{s\in I}Ch(g_{t,s})$. It  suffices to show that
\[
CS( \II(g_t) )= \int_{t\in S^1} CS(g_{t,s}) \quad \quad \textrm{mod exact,}
\]
which is the same as showing 
\[
CS([h]) + \eta = \int_{t\in S^1} \int_{s\in I}Ch(g_{t,s}) \quad \quad \textrm{mod exact.}
\]
for some choices (as in the definition of $\II$) of a smooth map $h: M \to BU \times \Z$ that is homotopic to $i \circ {}_* g_t$, and
a smooth map $k_s: M \times I \to \Om U$ such that $k_0 = E \circ h$ and $k_1 = {}_*g_t$, where, as always
\[
\eta = \int_s k_s^* CS + h^* \beta
\]

Since the map $g_t$ is nullhomotopic via $g_{t,s}$ from above, the map $i \circ {}_* g_t$ is also nullhomotopic via $i \circ {}_* g_{t,s}$ , so we may as well choose $h$ to be the constant map, which is smooth. Then $E \circ h$ is also the constant map (without loss of generality, to the identity). So, we may choose the homotopy $k_s = {}_* g_{t,s}$ and then  $k_0 = E \circ h$ is constant, and $k_1 = {}_* g_t$. Then $CS([h]) = 0$ and using Lemma \ref{lem;basedgt}
\begin{multline*}
CS([h])+\eta =
\eta = \int_s k_s^* CS = \int_s ({}_* g_{t,s})^* CS
\\
 = \int_s CS({}_* g_{t,s})  = \int_s CS(g_{t,s}) = \int_{t\in S^1} \int_{s\in I}Ch(g_{t,s}),
\end{multline*}
so we are done.
\end{proof}

\subsection{Even to Odd}\label{SEC:even-to-odd}
We now define an integration map $\II:\hat K^0(M\times S^1)\to \hat K^{-1}(M)$ and check the axioms for it. 

In the spirit of the last subsection, there is at least morally a ``wrong way'' map 
\[
H:Map(M, U) \to Map(M \times S^1,BU \times \Z)
\]
such that the following diagram 
\[
\xymatrix{
 Map(M \times S^1,BU \times \Z) \ar[r]^-{Ch}  &  \Om^{\textrm{even}}(M \times S^1) \ar[d]^{\int_{S^1}} \\
 Map(M, U)  \ar[u]^{H}  \ar[r]^-{Ch} &    \Om^{\textrm{odd}}(M)
}
\]
commutes for all compact manifolds with corners.  Geometrically, for $g: M \to U(n)=U(\C^n)$, consider the $\C^n$-bundle with connection $\nabla$ over $M \times S^1$ constructed by gluing the $\C^n$-bundle over $M \times I$ with connection $t g^{-1}dg$ for $t \in [0,1]$ along the endpoints $0$ and $1 \in [0,1]$. Then the diagram commutes in the sense that (\emph{c.f.} \cite[Equation (2.4)]{TWZ3})
\begin{equation}\label{EQU:int-Ch-nabla=Ch-g}
\int_{S^1} Ch(\nabla) = CS(t g^{-1}dg) = Ch(g)
\end{equation}
We use this idea as a guide in constructing an integration map in the correct direction.

Recall from Definition \ref{DEF:P-assoc-bundle} that a map $P_t: M \times S^1 \to \B$ induces a projection operator $P_t$ on a trivial $\C_p^q$ bundle over $M \times S^1$ (for some integers $p,q$), and an induced sub-bundle $E_{P_t}  \subset (M\times S^1)\times \C_p^q$, with connection $\nabla_{P_t}= P_t dP_t$, where $E_{P_t}$ is well defined up to adding a trivial bundle with trivial connection. To simplify the notation in this subsection, we will use the notation of denoting the Chern character $Ch(P_t)$ by
\[
Ch(\nabla_{P_t})=Ch(P_t).
\]
There is a map 
 \[
h_* : Map(M \times S^1, \B) \to Map(M, U)
\] 
given by
\[
h_*(P_t) = hol_{S^1}(\nabla_{P_t}) \oplus id
\]
where $hol_{S^1}(\nabla_{P_t}) \in End \left(E_{P_t} \big|_{M \times \{0\} }\right)$ is the holonomy of  $\nabla_{P_t}$ along the 
$S^1$ factor of $M \times S^1$, regarded as an endomorphism of the fiber of  $E_{P_t}|_{M\times \{0\}} \subset (M\times \{0\}) \times \C_p^q$ over $M \times \{0\} \subset M \times S^1$, and $id$ is the identity endomorphism on the orthogonal complement of $E_{P_t} \big|_{M \times \{0\} }$ in $(M\times \{0\}) \times \C_p^q$. Note that $h_*(P_t) : M \to  U$ is unchanged under adding a trivial bundle with trivial connection, since 
the holonomy of the trivial connection is the identity.

 The following diagram commutes in degree one, but does \emph{not} commute in general for higher degrees.
  \begin{equation} \label{eq;h_*diag}
\xymatrix{
 Map(M \times S^1, \B) \ar[d]^{h_*} \ar[r]^-{Ch}  &  \Om^{\textrm{even}}(M \times S^1) \ar[d]^{\int_{S^1}} \\
 Map(M, U)    \ar[r]^-{Ch} &    \Om^{\textrm{odd}}(M)
}
\end{equation}
In fact, here is an explicit example showing  the diagram does not commute in higher degrees, compare \cite{FL}.
 We show below (Lemma \ref{lem;deta2}) that the error in the diagram is always given by an exact form on $M$.

\begin{ex} By the Narasimhan Ramanan Theorem,  it suffices to produce an example of a bundle with connection for which the diagram does not commute.
Let $M = S^1 \times S^1 \times S^1$ and consider the trivial line bundle over $M \times S^1$ with connection $\nabla = i( A + B dt)$, using coordinates $(p,q,s,t) \in M \times S^1$. We verify that the diagram does not commute in degree three if $A = f(p) dq$ and $B = g(s)$ where $f$ and $g$ are such that $\frac{ \partial f}{\partial p} \frac{ \partial g}{\partial s}$ is not identically zero. For example, we may let $f(p) = \cos(p)$ and $g(s) = \sin(s)$.

First,  the degree three component of $Ch(hol_{S^1} (\nabla) )$ is zero for line bundles since if $g = hol_{S^1} (\nabla)= \exp( 2 \pi i \int B dt)$, then $(g^{-1}dg)^3 = 0$. On the other hand,  $-R = -\nabla^2 = d_M A + d_M B dt$, so that 
$R^2 = 2 d_M A \wedge d_M B dt$, and in degree three of $\int_{S^1} Ch( \nabla)$ we have
 \[
 \int_{S^1} R^2 = \int_{S^1} 2 d_M A \wedge d_M B dt = 2 \frac{ \partial f}{\partial p} \frac{ \partial g}{\partial s} dp dq ds.
 \]
 For example, $f(p) = \cos(p)$ and $g(s) = \sin(s)$, then in degree three we have $\int_{S^1} Ch( \nabla )  = -2 \sin(p) \cos(s) dp dq ds$.
 Note this is exact on $M$, which is the case in general, by Lemma \ref{lem;deta2} below. 
\end{ex}

Now, let $(E,\nabla)$ be the bundle with connection induced by $P_t \in Map(M\times S^1, \B)$. Let $\nabla_t$ be the connection on the slice of $E_t$ over $M \times \{t\}$ and let $g_t \in U(\C_{p}^q)$ be the parallel transport from $E_0$ to $E_t$. Then $g_0 = id$ and $g_1 = hol_{S^1}(\nabla) $. Note that $\nabla_0 = \nabla_1$. We can pullback the bundle $E \to M \times S^1$ along the projection $M \times I \to M \times S^1$ via the map $I\mapsto S^1$ which identifies the endpoints of $I$. By abuse of notation, we also denote by $(E,\nabla)$ the pullback bundle $E\to M\times I$ from $E\to M\times S^1$. Then we have,
\[
 \int_{S^1} Ch( \nabla)  =   \int_{I} Ch( \nabla).
 \] 
The latter term can be expressed equivalently using the following gauge transformation induced by parallel transport. There is a bundle isomorphism $g_t: E_{0} \times I \to E$ given by $(v,t) \mapsto g_t(v)$, so the bundle $E_{0} \times I$ over $M \times I$, with connection $g_t^* \nabla$, is isomorphic to the bundle $E \to M \times I$ with connection $\nabla$. Therefore
\[
 \int_{S^1} Ch( \nabla)  =   \int_{I} Ch( \nabla) = \int_I Ch(g_t^* \nabla).
 \] 
Finally, since the connection $g_t^* \nabla$ vanishes on $\partial / \partial t$, we can regard this as a path of connections on the 
bundle $E_{0} \to M$, and so 
\begin{equation}\label{EQU:int-ch-nabla=CS-g-nabla}
 \int_{S^1} Ch( \nabla)  =   \int_{I} Ch( \nabla) = \int_I Ch(g_t^* \nabla) = CS(g_t^* \nabla).
 \end{equation}
To see that diagram \eqref{eq;h_*diag} commutes up to an exact form, it now suffices to show that $CS(g_t^* \nabla_t)$ and $Ch( g_1)$ differ by an exact form that depends naturally on the map $P_t: M\times S^1 \to \B$ that determines $\nabla$. This is the form $\eta$ we now define.
\begin{defn} \label{defn;eta2} For a map $P_t: M \times S^1 \to \B$, let $\nabla$ be the induced connection on the sub-bundle $E \subset (M\times S^1)\times \C_p^q$, and let $\nabla^\perp$ be the induced connection on the  complementary bundle $E^\perp$. Define $\eta \in \Om^{\textrm{even}}(M)$ by
\begin{multline*}
\eta=\eta_{P_t}=  \int\int_{(r,t) \in I\times I}  Ch\Big( (1-r) \nabla_0 + r g_t^* \nabla_t \Big)
\\
+ \int\int_{(r,t) \in I\times I}  Ch \Big( (1-r) \left( t\nabla_0 \oplus t\nabla_0^\perp  \right) +  r (g_1 \oplus id)^*  \left( t\nabla_0 \oplus t\nabla_0^\perp  \right)\Big).
\end{multline*}
Note, that $\eta$ is given by integrating the Chern forms of two two-parameter families of connections over a square; see Figure \ref{rectofConns}. Note further, that since $\eta_{P_t}$ is determined by the connection $\nabla = P_tdP_t$ defined by the projection operator $P_t$, it follows that $\eta_{P_t}$ is natural in the sense that for a map $f:N \to M$, we have that $f^*(\eta_{P_t}) = \eta_{P_t\circ (f\times id)}\in \Om^{\even}(N)$.
\end{defn}

The definition of the $\eta$ form and some of its important properties are illuminated by the following lemma and its proof.
\begin{lem} \label{lem;deta2}
For a map $P_t: M \times S^1 \to \B$ we have
\[
d ( \eta_{P_t} ) =  \int_{S^1} Ch(P_t) - Ch(h_*(P_t)) 
\]
This shows that $h_*$ makes the diagram in \eqref{eq;h_*diag} commute modulo exact forms.
\end{lem}

\begin{proof} Let $\nabla$ be the connection induced by $P_t$.  The two terms in $\eta_{P_t}$ are induced by squares of connections as in Figure \ref{rectofConns}. We obtain $d(\eta_{P_t})$ from the sum of the boundaries of the two squares.
{\tiny 
\begin{figure}  
\begin{pspicture}(0,0)(12,5)
%\psgrid(0,0)(9,5)
\pspolygon(0,1)(3,1)(3,4)(0,4)
 \rput[r](0,1){$\nabla_0$}
 \rput[r](-.2,4){$\nabla_0$}
 \rput[l](3.1,1){$\nabla_0$}
 \rput[l](3.1,4){$g_1^* \nabla_1$}
 \rput[r](-.1,2.5){$\nabla_0$}
 \rput(1.5,.7){$\nabla_0$}
   \rput[l](3.2,2.5){$(1-r)\nabla_0 + r g_1^* \nabla_1$}
   \rput[l](.4,2.5){$(1-r)\nabla_0 + r g_t^* \nabla_t$}
   \rput[l](1.2,4.3){$g_t^* \nabla_t$}
 \pspolygon(9,1)(12,1)(12,4)(9,4)
 \rput[r](8.8,.8){$\nabla_0 \oplus \nabla_0^\perp$}
  \rput[r](8.8,4){$ (g_1 \oplus id)^* \left(\nabla_0  \oplus \nabla_0^\perp \right) $}
 \rput[l](12.1,1){$0$}
 \rput[l](12.1,4){$(g_1 \oplus id)^* (0)$}
 \rput[r](8.9,2.5){$\begin{matrix} (1-r) \left( \nabla_0 \oplus \nabla_0^\perp  \right)  \\ + \\r(g_1 \oplus id)^*  \left( \nabla_0 \oplus \nabla_0^\perp  \right)\end{matrix}$}
 \rput[r](12,2.5){$\begin{matrix} (1-r) \left( t\nabla_0 \oplus t\nabla_0^\perp  \right)  \\ +  \\r(g_1 \oplus id)^*  \left( t\nabla_0 \oplus t\nabla_0^\perp  \right)\end{matrix}$}
 \rput(10.5,.7){$ t \nabla_0  \oplus t \nabla_0^\perp $}
   \rput[l](12.1,2.5){$ r (g_1 \oplus id)^* (0)$}
   \rput[r](11.8,4.3){$(g_1 \oplus id)^* \left(t \nabla_0  \oplus t \nabla_0^\perp \right)$}
  \end{pspicture} 
\caption{The two parameter squares of connections defining $\eta$}
\label{rectofConns}
\end{figure}
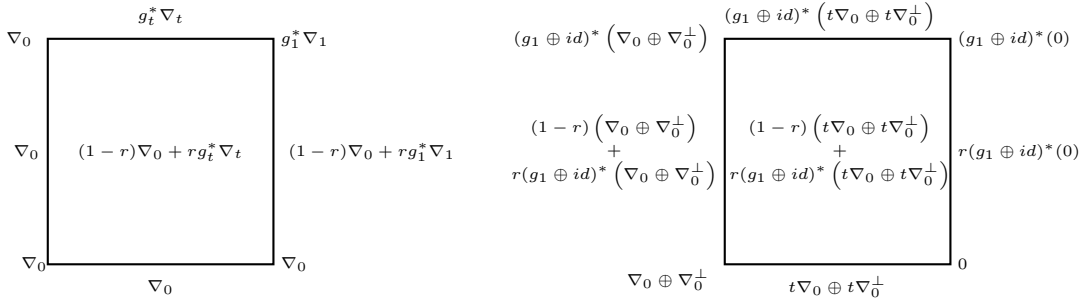
}

For the first term we have 
\[
d \left(\int\int_{(r,t) \in I\times I}   Ch\left( (1-r) \nabla_0 + r g_t^* \nabla_t \right) \right)= CS(g_t^* \nabla_t)  - CS ( (1-r)\nabla_0 + r g_1^* \nabla_1)
\]
since the other two boundary terms are constant paths of connections, so the Chern-Simons form vanishes. By equation \eqref{EQU:int-ch-nabla=CS-g-nabla} above we have
\[
CS(g_t^* \nabla_t) = \int_{I} Ch(g_t^* \nabla_t) =\int_{S^1} Ch( \nabla). 
\]

Calculating the boundary for the second square, the top and bottom edges $(g_1\oplus id)^*(t\nabla_0\oplus t\nabla_0^\perp)$ and $t\nabla_0\oplus t\nabla_0^\perp$ are gauge equivalent, so that they have equal Chern-Simons forms, which cancel, since they appear with opposite signs. Therefore,
\begin{multline*}
d \left(\int\int_{(r,t) \in I\times I}   Ch\left(  (1-r) \left( t\nabla_0 \oplus t\nabla_0^\perp  \right)  + r(g_1 \oplus id)^*  \left( t\nabla_0 \oplus t\nabla_0^\perp  \right) \right)\right)
 \\
=  CS\Big((1-r) \left( \nabla_0 \oplus \nabla_0^\perp  \right)  + r(g_1 \oplus id)^*  \left( \nabla_0 \oplus \nabla_0^\perp  \right)  \Big)- CS \Big(r  (g_1 \oplus id)^*(0)\Big).
\end{multline*}
Note, that the last term does not vanish, since $g_1^*(0)=g_1^{-1}dg_1$, but it is
\[
CS (r  (g_1 \oplus id)^*(0)) = CS( r ( g_1^{-1} dg_1 \oplus 0)) = CS( r  g_1^{-1} dg_1) \stackrel{\eqref{EQU:int-Ch-nabla=Ch-g}}
= Ch(g_1)= Ch(h_*(P_t)).
\] 
The remaining two $CS$ terms cancel, because
\[
(1-r) \left( \nabla_0 \oplus \nabla_0^\perp  \right)  + r(g_1 \oplus id)^*  \left( \nabla_0 \oplus \nabla_0^\perp  \right) =
 \Big( (1-r)\nabla_0 + r g_1^* \nabla_0  \Big)  \oplus \nabla_0^\perp,
\]
and since the right summand $\nabla_0^\perp$ is time independent, it does not contribute to the $CS$ form, so that using $\nabla_0=\nabla_1$, we have that
\[
CS\Big((1-r) \left( \nabla_0 \oplus \nabla_0^\perp  \right)  + r(g_1 \oplus id)^*  \left( \nabla_0 \oplus \nabla_0^\perp  \right)\Big) =
 CS\Big( (1-r)\nabla_0 + r g_1^*\nabla_1  \Big).
\]
This shows that $d\eta=\int_{S^1} Ch(\nabla)-Ch(h_*(P_t))$, which is the claim of the lemma.
\end{proof}
We define the $S^1$-integration map $\II$ using this error of making diagram \eqref{eq;h_*diag} commute, \emph{c.f.} Definition \ref{defn;S1int}.
\begin{defn}
We define $\II:\hat K^{0}(M\times S^1)\to \hat K^{1}(M)$ for a representative $P_t: M \times S^1 \to \B$ by setting
\[
\II(P_t) = [h_*(P_t)] + a(\eta_{P_t})
\]
where $[h_*(P_t)] \in \hat K^{-1}(M)$ is the equivalence class of $h_*(P_t)$.  
\end{defn}

In the following Lemmas \ref{lem;welldefn2}-\ref{LEM:I-versus-a-map-even-to-odd}, we will show that $\II$ is well-defined, and satisfies the axioms for the integration map from Definition \ref{defn:S^1int}.
\begin{lem} \label{lem;welldefn2}
The map $\II : \hat K^{0}(M \times S^1) \to \hat K^{-1}(M)$ is well defined.
\end{lem}
\begin{proof}
Let $P_{t,0}, P_{t,1} : M \times S^1 \to BU \times \Z $ be given and $CS$-equivalent, so there is a path $P_{t,s} : M \times S^1 \times I \to BU \times \Z $ such that
\[
CS(P_{t,s})=\int_{s \in I} Ch(P_{t,s}) \in \Om^{\odd}(M \times S^1)
\]
is exact.  Note this assumption implies that 
\[
\int_{t \in S^1} \int_{s \in I} Ch(P_{t,s}) = \int_{s \in I}  \int_{t \in S^1}  Ch(P_{t,s}) 
\]
is also exact. 

We show that 
\[
[h_*(P_{t,1})]- [h_*( P_{t,0})]  = a( \eta_{P_{t,0}}) - a( \eta_{P_{t,1}}). 
\]
Since $h_*(P_{t,0})$ and $h_*(P_{t,1})$ are homotopic via $s\mapsto h_*(P_{t,s})$, the class $[h_*(P_{t,1})]- [h_*( P_{t,0})]  $ is in the kernel of $I$. So it suffices to show that the path $s \mapsto h_*(P_{t,s})$ satisfies
\[
CS(h_*(P_{t,s})) = \eta_{P_{t,1}} - \eta_{P_{t,0}}  \quad \quad \textrm{mod exact forms}.
\]
and that
\[
CS([h_*(P_{t,1})]- [h_*( P_{t,0})]) = CS(h_*(P_{t,s}))\quad \quad \text{ mod } Im(Ch).
\]

The idea for showing the first equality is to consider the data in Figure \ref{rectofConns} varying smoothly with parameter $s \in [0,1]$, and construct a differential form whose exterior derivative is the difference $CS(h_*(P_{t,s})) + \eta_{P_{t,1}} - \eta_{P_{t,0}} $. Let
\[
\omega = \int_{s \in I}  \eta_{P_{t,s}}.
\]
Then, with $d \eta_{P_{t,s}} = \int_{t\in S^1} Ch(P_{t,s}) - Ch( h_*(P_{t,s}))$, we have
\begin{equation*}
d \omega =    \eta_{P_{t,1}} -  \eta_{P_{t,0}}  - \int_{s \in I}  d \eta_{P_{t,s}}  
=  \eta_{P_{t,1}} -  \eta_{P_{t,0}}  - \int_{s \in I} \left( \int_{t\in S^1} Ch( P_{t,s} ) - Ch(h_*(P_{t,s} )) \right)
\end{equation*}
But $\int_{t\in S^1} Ch( h_*(P_{t,s} ) = CS(h_*(P_{t,s}))$ and $\int_{s \in I} \int_{t\in S^1} Ch( P_{t,s} )$ is exact by assumption.

It remains to show $CS([h_*(P_{t,1})]- [h_*( P_{t,0})]) = CS(h_*(P_{t,s})) $ mod $Im(Ch)$. 
There is a path $\gamma_s$ from $h_*(P_{t,0}) \oplus h_*( P_{t,0})^{-1}$  to $0$ such that $CS(\gamma_s)= 0$ (by the proof of the existence of inverses). Composing this with the path $h_*(P_{t,s}) \oplus id$, from $h_*(P_{t,1}) \oplus h_*(P_{t,0})^{-1}$ to $h_*(P_{t,0}) \oplus h_*(P_{t,0})^{-1}$, we obtain a path $\left(h_*(P_{t,s}) \oplus id \right)  *  \gamma_s$ from $h_*(P_{t,1}) -h_*(P_{t,0})=h_*(P_{t,1}) \oplus h_*(P_{t,0})^{-1}$ to $0$ whose $CS$-form equals $CS(h_*(P_{t,s})) + CS(\gamma_s) = CS(h_*(P_{t,s})) $. This shows that 
\[
CS(h_*(P_{t,s})) = CS([h_*(P_{t,1})] -[h_*(P_{t,0})])  \quad \quad \textrm{mod $Im(Ch)$}.
\] 
This completes the proof of the lemma.
\end{proof}

\begin{lem} \label{lem;grouphomo2}
The map $\II : \hat K^{0}(M \times S^1) \to \hat K^{-1}(M)$ is a group homomorphism.
\end{lem}

\begin{proof} Let $P_{t,1}, P_{t,2} : M \times S^1 \to \B$. 
By re-ordering we have
\[
h_*(P_{t,1} \oplus P_{t,2}) = hol_{S^1} (P_{t,1} \oplus P_{t,2}) \oplus id
\]
is CS-equivalent to 
\[
h_*(P_{t,1}) \oplus h(P_{t,2}) = hol_{S^1} (P_{t,1}) \oplus id \oplus  hol_{S^1} (P_{t,2}) \oplus id 
\]
so $[h_*(P_{t,1} \oplus P_{t,2})] = [h_*(P_{t,1})] + [h_*(P_{t,2})]$. Also, 
\[
\eta_{P_{t,1}\oplus P_{t,2}} = \eta_{P_{t,1}} + \eta_{P_{t,2}}
\]
since the induced connections' parallel transport, and the Chern Character, respect block sum. So, we conclude that $\II$ is a group homomorphism.
\end{proof}

\begin{lem}
The map $\II : \hat K^{0}(M \times S^1) \to \hat K^{-1}(M)$ satisfies $\II \circ ( id \times r)^* = - \II$ where $r :S^1 \to S^1$ is given by $r(z) = \bar z = z^{-1}$. 
\end{lem}
\begin{proof}
Since holonomy along the reverse loop gives the inverse group element, we have, for $P_t:M\times S^1\to \B$,
\[
 [h_*((id \times r)^* P_t)] = [(h_*(P_t))^{-1}] = - [(h_*(P_t)].
 \]
 Secondly, since $a$ is linear, it suffices to show that
$\eta_{(id \times r)^* P_t} = - \eta_{P_t}$.  Recall 
\begin{multline*}
\eta_{P_t}=  \int\int_{(s,t) \in I\times I}  Ch\left( (1-s) \nabla_0 + s g_t^* \nabla_t \right)
\\
+ \int\int_{(s,t) \in I\times I}  Ch \left( (1-s) \left( t\nabla_0 \oplus t\nabla_0^\perp  \right) +  s (g_1 \oplus id)^*  \left( t\nabla_0 \oplus t\nabla_0^\perp  \right)\right),
\end{multline*}
where for the purposes of the integral the variable  $t\in I$ can also be regarded as $t \in S^1$. In this way, the map $r: S^1 \to S^1$ corresponds to the substitution $t \mapsto 1-t$ for $t \in I$. Making this substitution in the equation above yields $- \eta_{(id \times r)^* P_t}$, which shows the claim.
\end{proof}

\begin{lem}
The  map $\II: \hat K^{0}(M \times S^1) \to \hat K^{-1}(M)$ satisfies $\II \circ p^* = 0$, where $p : M \times S^1 \to M$ is the projection.
\end{lem}
\begin{proof}
Let $P:M\times S^1\to \B$ be a representative for $[P] \in \hat K^{0}(M)$. Then $P_t : M \times S^1 \to \B$ defined by $P_t = p^*(P)$ satisfies $P_t = P$ for all $t \in S^1$. Therefore the connection $\nabla$ on $E \to M \times S^1$ induced by $P_t$  is zero in the $S^1$ direction, \emph{i.e.} $\nabla_{\partial / \partial t} = 0$, and that $\nabla_t=\nabla_0$ for all $t$. It follows that the parallel transport in the $S^1$ direction satisfies $g_t = id$ for all $t$, and in particular, the holonomy along the $S^1$ direction satisfies $h_*(p^*(P)) = id$. Also, the two squares of connections defining $\eta_{p^*(P)}$, as in Figure \ref{rectofConns}, do not depend on $r$, since
\[
 (1-r) \nabla_0 + r g_t^* \nabla_t =  \nabla_0
\]
and
\[
(1-r) \left( t\nabla_0 \oplus t\nabla_0^\perp  \right) +  r (g_1 \oplus id)^*  \left( t\nabla_0 \oplus t\nabla_0^\perp  \right)
=  t\nabla_0 \oplus t\nabla_0^\perp 
\]
for all $r$. Thus, $\eta_{p^*(P)} = 0$. It follows that $a(\eta_{p^*(P)}) = 0$ and so 
\[
\II(p^*(P)) = [h_*(p^*(P))] +  \eta_{p^*(P)} = 0.
\]
This completes the proof of the lemma.
\end{proof}

\begin{lem} \label{lem;Chdiag2}
The map $\II: \hat K^{0}(M \times S^1) \to \hat K^{-1}(M)$ makes the following diagram commute 
\[
\xymatrix{
\hat K^{0}(M \times S^1) \ar[r]^{Ch} \ar[d]^{\II} &  \Om^{\textrm{even}}(M \times S^1) \ar[d]^{\int_{S^1}} \\
 \hat K^{-1} (M)  \ar[r]^-{Ch} &    \Om^{\textrm{odd}}(M)
}
\]
\end{lem}
\begin{proof}
Given a representative $P_t: M \times S^1 \to \B$ for $[P_t] \in \hat K^0(M \times S^1)$, we have, using $Ch\circ a = d$, that
\begin{eqnarray*}
Ch ( \II( P_t) ) 
& =& Ch  (h_*(P_t) ) + Ch( a(\eta_{P_t})) \\
 &=& \left( \int_{S^1} \, Ch( P_t ) -   d  \eta_{P_t}  \right) + d \eta_{P_t}  
 = \int_{S^1}\, Ch (P_t).
 \end{eqnarray*}
This is the claim of the lemma.
\end{proof}

\begin{lem}
The map $\II: \hat K^{0}(M \times S^1) \to \hat K^{-1}(M)$ makes the following diagram commute 
\[
\xymatrix{
\hat K^{0}(M \times S^1) \ar[r]^{I} \ar[d]^{\II} &  K^{0}(M \times S^1) \ar[d]^{\int_{S^1}} \\
 \hat K^{-1} (M)  \ar[r]^-{I} &  K^{-1} (M) 
}
\]
\end{lem}
\begin{proof}
Since $Im(a) \subset Ker(I)$, it suffices to show that if $P_t : M \times S^1 \to \B$, then
\begin{equation}\label{EQU:I-II-odd}
 \int_{S^1} I( [P_t] ) = I( h_*(P_t) ).
\end{equation}

Recall that the $S^1$-integration map in ordinary $K$-theory is defined by the lower row in the following diagram,
\[
\xymatrix{
 [M \times S^1, \B] \ar[rrr]^-{h_*} \ar[d]^{=} & & & [M, U] \ar[d]^{=} \\
K^{0}(M \times S^1) \ar[r]^-{pr} & Ker(j^*) \ar[r]^-{(q^*)^{-1}} & \tilde K^{0}(\Sigma M_+) =K^{0}(M\wedge S^1) \ar[r]^-{\sigma^{-1}} & K^{-1}(M)
}
\]
where $h_*(P_t) = hol_{S^1}(P_t) \oplus id$ is the map induced by holonomy in the $S^1$ direction, and $I( h_*(P_t) ): M \to U$ is the homotopy class of the holonomy. The lower row can be interpreted geometrically as follows. If $E$ is the bundle over $M \times S^1$ determined by $P_t$, then $pr(E) =  E \oplus p^*j^*E^\perp$ is the bundle over $M \times S^1$ obtained by adding to $E$ the bundle $p^*E_0^\perp$ over $M \times S^1$, where $p: M\times S^1 \to M$ is projection, and $j^*E$ is the restriction of $E$ to $M \times \{0\}$. This bundle $pr(E)$ is trivial over $M \times S^1$, and the induced bundle over $\Sigma M_+$ is determined by (a homotopy class of) a map $M \to U$, given by the clutching map, which is determined by trivializing the bundle over each half of the suspension. But, this clutching map is given up to homotopy by $h_*(P_t) = hol_{S^1}(P_t) \oplus id$ since parallel transport of the connection trivializes the bundle over each cone. This shows that the above diagram commutes up to homotopy, and thus we have equation \eqref{EQU:I-II-odd}. This completes the proof.
\end{proof}

\begin{lem}\label{LEM:I-versus-a-map-even-to-odd}
The map $\II: \hat K^{0}(M \times S^1) \to \hat K^{-1}(M)$ makes the following diagram commute 
\[
\xymatrix{
\Om^{\odd}(M \times S^1)/Im(d) \ar[r]^-{a} \ar[d]^{\int_{S^1}} & \hat K^{0}(M \times S^1)\ar[d]^{\II}\\
\Om^{\even}(M) \ar[r]^-{a} /Im(d)& \hat K^{-1}(M ) 
}
\]
\end{lem}
\begin{proof}
By definition of the maps $a:\Om^{*+1}(M)/Im(d) \to \hat K^*(M)$, we can write the diagram in the statement of the lemma as 
\[
\xymatrix{
\Om^{\odd}(M \times S^1)/Im(d) \ar[r]^-{pr} \ar[d]^{\int_{S^1}} &  \Om^{\odd}(M \times S^1)/Im(Ch) 
\ar[d]^{\int_{S^1}} \ar[r]^-{CS^{-1}} \ar[d]^{\int_{S^1}}   &  Ker(I) \subset \hat K^{0}(M \times S^1)\ar[d]^{\II}\\
\Om^{\even}(M) \ar[r]^-{pr} /Im(d) & \Om^{\even}(M) \ar[r]^-{CS^{-1}} /Im(Ch)  &  Ker(I) \subset  \hat K^{-1}(M ) 
}
\]
where $pr$ is projection, which is well defined since $Im(d) \subset Im(Ch)$ by Theorem \ref{THM:3-statements}\eqref{statement-1}, and the middle vertical map is well defined since $Im(\int_{S^1} \circ Ch) \subset Im(Ch)$ by Lemma \ref{lem;Chcommutes}. The left square clearly commutes so it suffices to show the right square commutes. The maps $CS^{-1}$ are isomorphisms onto $Ker(I)$.

Given $P_t: M \times S^1 \to BU \times \Z$ such that $P_t \in Ker (I)$ we can choose $P_{t,s} :M \times S^1 \times I \to BU \times \Z$ such that $P_{t,1} = P_t$ and $P_{t,0}$ is constant. We need to show that
\[
CS(\II(P_t)) = \int_{t \in S^1} CS(P_{t,s})  \quad \quad \textrm{mod exact forms}.
\]
The class $[h_*(P_{t,1})]$ is also in the kernel of $I$, and $h_*(P_{t,s})$ is a path from $h_*(f_{P,1})$ to the constant $h_*(P_{t,0})$.
So it suffices so to show
\[
CS(h_*(P_{t,s})) + \eta_{P_{t,1}} = \int_{t \in S^1} CS(P_{t,s})  \quad \quad \textrm{mod exact forms}.
\]

As in Lemma \ref{lem;welldefn2}, let
\[
\omega = \int_{s \in I}  \eta_{P_{t,s}}.
\]
Using $d \eta_{P_{t,s}} = \int_{t\in S^1} Ch(P_{t,s}) - Ch( h_*(P_{t,s}))$ we have
\[
d \omega  =  \eta_{P_{t,1}} -  \eta_{P_{t,0}}  - \int_{s \in I}\int_{t\in S^1} Ch( P_{t,s} ) +  \int_{s \in I} Ch(h_*(P_{t,s} )).
\]
Now $\eta_{P_{t,0}} = 0$ since $P_{t,0}$ is constant, $\int_{s \in I} \int_{t\in S^1} Ch( P_{t,s} ) = \int_{t \in S^1} CS(P_{t,s})$, and $\int_{t\in S^1} Ch( h_*(P_{t,s} ) = CS(h_*(P_{t,s}))$, which shows that
\[
d\omega=  \eta_{P_{t,1}} - \int_{t \in S^1} CS(P_{t,s})+  CS(h_*(P_{t,s})).
\]
This completes the proof of the lemma.
\end{proof}

The collection of lemmas in subsections \ref{SEC:odd-to-even} and \ref{SEC:even-to-odd} prove the following corollary.
\begin{cor}
The map $\II: \hat K^{*+1}(M \times S^1) \to \hat K^*(M)$ defines an $S^1$-integration map.
\end{cor}

Finally, by the uniqueness theorem of differential $K$-theory (Theorem 3.3 of \cite{BS3}) we have

\begin{thm}
Differential $K$-theory is represented by 
\[
\hat K^0(M) = Hom(M, \B) \quad \quad \hat K^{-1}(M) = Hom(M, U)
\]
as functors on $Smooth_{\hat K}$ (see Definition \ref{defn:SmoothK}) where $Hom$ means $CS$-equivalence classes of maps
 (see Definitions \ref{defn:CSequiveven} and \ref{defn:CSequivodd}).
\end{thm}

\end{document}